\documentclass[11pt]{article}

\usepackage{mathpazo}
\usepackage{amsfonts}
\usepackage{amsmath}
\usepackage{graphicx}
\usepackage{latexsym}
\usepackage{mathabx}
\usepackage{mathtools}

\usepackage[margin=1.05in]{geometry}

\usepackage[T1]{fontenc}
\usepackage{times}
\usepackage{color,graphicx}
\usepackage{array}
\usepackage{enumerate}
\usepackage{amsmath}
\usepackage{amssymb}
\usepackage{amsthm}
\usepackage{pgfplots}
\usepackage{pgf}
\usepackage{tikz}
\usetikzlibrary{patterns}
\usepgfplotslibrary{patchplots} 
\usetikzlibrary{pgfplots.patchplots} 
\pgfplotsset{width=9cm,compat=1.5.1}

\usepackage{amsthm}
\theoremstyle{plain}
\newtheorem{thm}{Theorem}[section]
\newtheorem{theorem}[thm]{Theorem} 
\newtheorem{lemma}[thm]{Lemma} 
\newtheorem{proposition}[thm]{Proposition}
\newtheorem{conjecture}[thm]{Conjecture}
\newtheorem{corollary}[thm]{Corollary}

\theoremstyle{definition}
\newtheorem{example}[thm]{Example}

\newtheorem{definition}[thm]{Definition}
\newtheorem{remark}[thm]{Remark}

\usepackage{xcolor}
\usepackage{makeidx}
\usepackage[colorlinks=true,linkcolor=blue,anchorcolor=blue,citecolor=red,urlcolor=magenta]{hyperref}
\usepackage[backrefs]{amsrefs}

\newcommand{\cT}{\mathcal T}
\DeclarePairedDelimiter{\ceil}{\lceil}{\rceil}
\DeclareMathOperator{\diam}{diam}
\DeclareMathOperator{\dist}{dist}
\DeclareMathOperator{\Kf}{Kf}

\title{Bounds on Kemeny's constant of a graph and the Nordhaus-Gaddum problem}

\author{
Sooyeong Kim\textsuperscript{1}\footnote{Contact: kimswim@yorku.ca} \and
Neal Madras\textsuperscript{1} \and
Ada Chan\textsuperscript{1} \and
Mark Kempton\textsuperscript{2} \and 
Stephen Kirkland\textsuperscript{3} \and
Adam Knudson\textsuperscript{2}
}
\date{\today}

\begin{document}

\maketitle

\begin{abstract}
    We study Nordhaus-Gaddum problems for Kemeny's constant $\mathcal{K}(G)$ of a connected graph $G$.  We prove bounds on $\min\{\mathcal{K}(G),\mathcal{K}(\overline{G})\}$ and the product $\mathcal{K}(G)\mathcal{K}(\overline{G})$ for various families of graphs.  In particular, we show that if the maximum degree of a graph $G$ on $n$ vertices is $n-O(1)$ or $n-\Omega(n)$, then $\min\{\mathcal{K}(G),\mathcal{K}(\overline{G})\}$ is at most $O(n)$.
\end{abstract}



\addtocounter{footnote}{1}
\footnotetext{Department of Mathematics and Statistics, York University, 4700 Keele Street, Toronto, Ontario, Canada}
\addtocounter{footnote}{1}
\footnotetext{Department of Mathematics, Brigham Young University, Provo UT, USA}
\addtocounter{footnote}{1}
\footnotetext{Department of Mathematics, University of Manitoba, Winnipeg, Manitoba, Canada}

\section{Introduction}


Kemeny's constant is an important measure from the theory of Markov chains that has received considerable interest from the graph theory community recently. It also arises as a tool in applications of Markov chains in diverse areas such as wireless network design \cite{GhaL} and  economics \cite{MooI}. 
Kemeny's constant is originally defined for a discrete, finite, time-homogeneous, irreducible Markov chain based on its stationary vector and mean first passage times. Random walks on graphs belong to this special family of Markov chains, and they serve as the primary focus in this article. Consequently, our attention focuses on Kemeny's constant within the context of random walks on graphs, and we refer the reader to \cite{kemeny1960finite} for the original definition and details.  Kemeny's constant gives a measure of how quickly a random walker can move around a graph, and  thus provides an intuitive measure of the connectivity of a graph.  

Let $G$ be a graph with vertex set $V(G)$ and edge set $E(G)$. In this article, we assume all graphs to be connected, undirected, and unweighted unless stated otherwise.  Let $V(G)=\{1,\dots,n\}$ for some $n\geq 1$, and $m=|E(G)|$. We denote by $\{i,j\}$ the edge joining vertices $i$ and $j$, and say that $i$ and $j$ are adjacent. We denote by $d_i$ the degree of vertex $i$; that is, the number of vertices adjacent to vertex $i$.

We begin by introducing Kemeny's constant and discussing its interpretation in the context of random walks on graphs, and discuss several other expressions for Kemeny's constant which will be used throughout the paper. Given a graph $G$ and fixed initial vertex $i\in V(G)$ from which to start the random walk, Kemeny's constant $\mathcal{K}(G)$ is defined as
\begin{align}\label{eq:def}
	\mathcal{K}(G) = \sum_{\substack{j=1 \\ j\neq i}}^n \left(\frac{d_j}{2m}\right)m_{i,j},
\end{align}
where $m_{i,j}$ is the expected time for a random walker to arrive at $j$ for the first time when it begins at $i$, which is the so-called \textit{mean first passage time} or \textit{hitting time} from $i$ to $j$. This expression can be interpreted in terms of the expected time for a random walk to reach a randomly-chosen destination vertex, starting from vertex $i$. The `constant' in the name comes from the fact that this expression is, astonishingly, independent of the choice of $i$, giving the same value regardless of the fixed start point. We note that $\sum_i \frac{d_i}{2m} = 1$, and $\frac{d_i}{2m}$ is the stationary probability distribution of the random walk. Furthermore, it follows that \eqref{eq:def} may be rewritten as
\begin{align*}
	\mathcal{K}(G) = \sum_{i=1}^n\sum_{\substack{j=1\\j\neq i}}^n\left(\frac{d_i}{2m}\right)m_{i,j}\left(\frac{d_j}{2m}\right).
\end{align*}
Hence, Kemeny's constant may also be interpreted in terms of the average travel time between two vertices of the graph, chosen at random according to their degrees.



It is shown in \cite{levene2002kemeny} that Kemeny's constant for a Markov chain can be written in terms of the eigenvalues of the transition matrix; specifically,
\begin{align}\label{eq:eig}
	\mathcal{K}(G) = \sum_{i=2}^n \frac{1}{1-\lambda_i},
\end{align}
where $1,\lambda_2,\dots,\lambda_n$ are the eigenvalues of the associated  transition matrix. In the particular case of a random walk on $G$, the  matrix $D^{-1}A$ is the transition matrix, where $D$ is the diagonal matrix of vertex degrees, and $A$ is the adjacency matrix of $G$ 
 (see Equation (\ref{eq.Pvw})).

For a random walk on a graph, Kemeny's constant also has a combinatorial expression which we use throughout this article. Let $\mathcal{F}(i;j)$ denote the set of spanning $2$-forests of $G$ where one component of the forest contains vertex $i$, and the other contains $j$. Let $F$ be the matrix given by $F = [f_{i,j}]$ where $f_{i,j} = |\mathcal{F}(i;j)|$. In \cite{kirkland2016kemeny}, Kemeny's constant of $G$ is given by
\begin{align*}
	\mathcal{K}(G) = \frac{\mathbf{d}^T F \mathbf{d}}{4m\tau} = \frac{1}{4m\tau}\sum_{i=1}^n\sum_{\substack{j=1\\j\neq i}}^n d_if_{i,j}d_j,
\end{align*}
where $\mathbf{d}$ is the degree vector of $G$ and $\tau$ is the number of spanning trees of $G$. Let $R$ be the matrix given by $R = [r_{i,j}]$, where $r_{i,j}$ is the \textit{effective resistance} (see \cite{bapat2010graphs}) between vertices $i$ and $j$. The quantity $r_{i,j}$ is given by $r_{i,j}=(e_i-e_j)^TL^\dagger (e_i-e_j)$ where $L^\dagger$ is the Moore-Penrose inverse of the Laplacian matrix $L$ of $G$, which is given by $L = D-A$. It appears in \cite{shapiro1987electrical} that $r_{i,j} = \frac{f_{i,j}}{\tau}$. Hence, we also have
\begin{align}\label{eq:res}
	\mathcal{K}(G) = \frac{\mathbf{d}^T R \mathbf{d}}{4m} = \frac{1}{4m}\sum_{i=1}^n\sum_{j=1}^n d_ir_{i,j}d_j.
\end{align}

In this paper, we will be concerned with Nordhaus-Gaddum questions relating to Kemeny's constant.  Nordhaus-Gaddum questions in graph theory are questions that address the relationship between a graph $G$ and its complement relative to some graph invariant. The complement of a graph $G$, denoted $\overline{G}$, is the graph with the same vertex set $V(G)$ such that $\{i,j\}\in E(\overline{G})$ if and only if $\{i,j\} \notin E(G)$. Given some graph invariant $f(\cdot)$, a Nordhaus-Gaddum question usually considers bounds on either the sum $f(G)+f(\overline{G})$ or the product $f(G)f(\overline{G})$, or on the minimum or maximum of the set $\{f(G), f(\overline{G})\}$. Nordhaus and Gaddum originally studied such questions for the chromatic number of a graph \cite{nordhaus1956complementary}.  Since then, Nordhaus-Gaddum questions have been studied for a wide range of different kinds of graph invariant; see \cite{aouchiche2013survey} for a survey of results of this kind.  Of particular note to the present work, Nordhaus-Gaddum questions relating to a graph invariant called the \emph{Kirchhoff index} of a graph were studied in \cite{zhou2008note,yang2011new}.  The Kirchhoff index $\Kf(G)$ is defined as
\[\Kf(G) =\tfrac{1}{2}\sum_{i=1}^n \sum_{j=1}^n r_{i,j}.
\]
In light of the formula in \eqref{eq:res}, we can view Kemeny's constant as a weighted, normalized version of the Kirchhoff index.  Work in \cite{yang2011new} gave upper and lower bounds on the sum and product of $\Kf(G)$ and $\Kf(\overline G)$. In particular, they show that for a graph on $n$ vertices, the product $\Kf(G) \Kf(\overline{G})$ is at least $4(n-1)^2$ and is at most on the order of $n^4$, though this upper bound is not shown to be tight.  The lower bound was improved to $(2n-1)^2$ in \cite{das2016nordhaus}.
In \cite{chen2007resistance}, the quantity \[\Kf^\ast(G) = \frac{1}{2}\sum_{i=1}^n\sum_{j=1}^n d_id_jr_{i,j}\] is introduced, and is called the \textit{multiplicative degree-Kirchhoff index}. This is closely related to Kemeny's constant via the formula in \eqref{eq:res}. A Nordhaus-Gaddum question for this quantity was considered (among other related quantities) in \cite{das2016nordhaus}; in particular, certain bounds are given on the sum $\Kf^\ast(G)+\Kf^\ast(\overline{G})$. In \cite{faught2023nordhaus}, Nordhaus-Gaddum questions were also studied for the spectrum of the normalized Laplacian matrix defined by $D^{-\frac{1}{2}}LD^{-\frac{1}{2}}$, which is closely related to Kemeny's constant via \eqref{eq:eig}.

For a connected graph $G$, it is proved in \cite{breen2019computing} that $\mathcal{K}(G) = O(n^3)$. Thus it follows trivially that $\mathcal{K}(G)+\mathcal{K}(\overline{G}) = O(n^3)$. Consequently, it is more natural to examine the product $\mathcal{K}(G)\mathcal{K}(\overline{G})$ for the Nordhaus-Gaddum problem in relation to Kemeny's constant. We note that the motivation for considering this type of problem is to further understand the influence of graph structural features on the value of Kemeny's constant, given that extremal graphs tend to have large diameter and maximum degree on the order of $n$. 


To begin our investigation of Nordhaus-Gaddum questions for Kemeny's constant, we examine some data.  In Figure \ref{fig:kemplotswithbounds}, we have plotted the point $(\mathcal{K}(G),\mathcal{K}(\overline{G}))$ for all graphs $G$ on $n$ vertices for $n=7,8,9$.  For disconnected graphs, we consider Kemeny's constant to be infinite and plot the point at the extreme ends of the plot.  

 \begin{figure}[h]
    \centering
    \includegraphics[clip, trim ={0 3cm 0 3cm}]{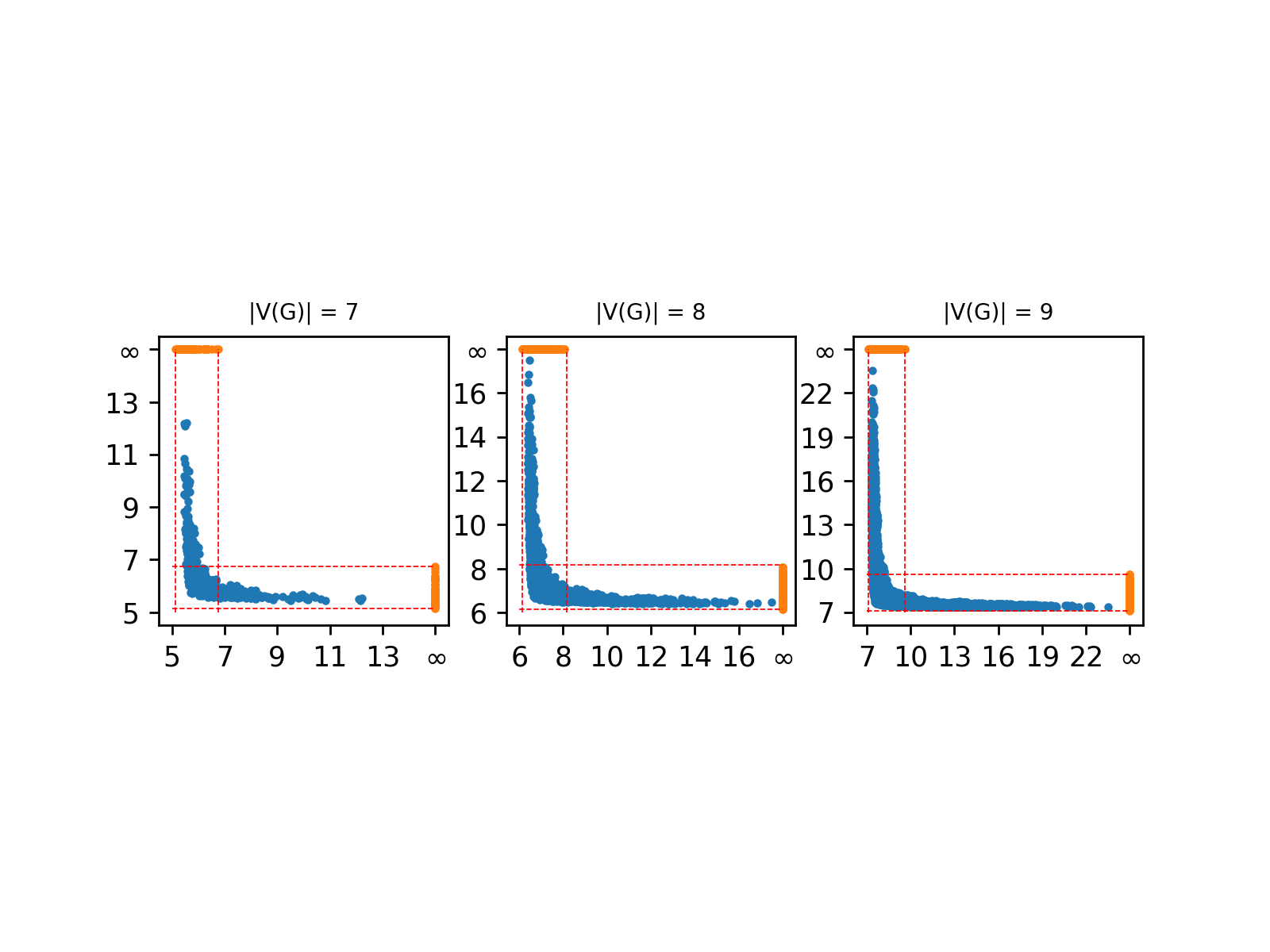}
    \caption{Scatter plots of $(\mathcal{K}(G), \mathcal{K}(\overline{G}))$ for all graphs of order $n = 7, 8, 9$. The red lines mark the quantities $\frac{(n-1)^2}{n}$ and $\frac{3(n-1)^2}{2(n+1)}$.}
    \label{fig:kemplotswithbounds}
\end{figure}

Looking at these plots, a few observations become apparent.  It seems that when Kemeny's constant for a graph is particularly large, then Kemeny's constant for its complement stays rather small.  In the plots, we have indicated the line $y=(n-1)^2/n$, $x=(n-1)^2/n$ which come from the known minimum value of Kemeny's constant for a graph of order $n$, achieved by a complete graph.  We also have the lines $y=3(n-1)^2/2(n+1)$ and $x=3(n-1)^2/2(n+1)$ which come from computation of Kemeny's constant for graphs that appear at the extreme ends of the plots.  For all graphs on up to $n=9$ vertices, it appears that all points lie in the regions bounded by these four lines.  While it is difficult to make a general conjecture based on looking at graphs on only up to nine vertices, it does seem natural to ask if one of $\mathcal{K}(G)$ or $\mathcal{K}(\overline{G})$ must be bounded above by $O(n)$.  If true, this would imply that $\mathcal{K}(G)\mathcal{K}(\overline{G})$ is $O(n^4)$.

In this paper, we prove this and related results under certain assumptions and for certain classes of graphs. In Section \ref{sec:n-O1} we prove that if the maximum degree of $G$ is $n-O(1)$, then $\mathcal{K}(G)$ is $\Theta(n)$ and hence $\mathcal{K}(G)\mathcal{K}(\overline{G})=O(n^4)$.  In Section \ref{sec:n-on}, we show that if the maximum degree of $G$ is at most $nU$ for some $U<1$, then 
\[\min\{\mathcal{K}(G),\mathcal{K}(\overline{G})\} \leq Cn
\] 
for some constant $C$ (possibly depending on $U$).  An important corollary is that the same inequality is true for all regular graphs, but now with $C$ independent of the degree.  In Section \ref{sec:join}, we examine joins of graphs (i.e. graphs whose complements are disconnected) and show that for any such graph $\mathcal{K}(G)\leq 3n$.   Finally, in Section \ref{sec:app}, we examine particular families of graphs, and determine more specific bounds on $\mathcal{K}(G)\mathcal{K}(\overline{G})$ for various families such as regular graphs, barbell graphs, trees, strongly regular graphs and certain distance regular graphs with classical parameters.  We end with some concluding remarks and a conjecture in Section \ref{sec:concl}.

\subsection{Notation and preliminaries}
     \label{sec:notation}
We introduce the following notation and definitions which will be used throughout the remainder of this article. Given a graph $G$, we let $\delta(G)$ denote the minimum degree of a vertex of $G$, and $\Delta(G)$ denote the maximum degree. We denote by $N(i)$ the \emph{neighbourhood} of $i$; i.e. the set of vertices adjacent to vertex $i$. The distance between two vertices $i$ and $j$, denoted $\dist(i,j)$, is defined as the length of the shortest path between vertices $i$ and $j$, while the \emph{diameter} of $G$ is the longest pairwise distance in the graph:
\[\diam(G) = \max_{i, j \in V(G)} \dist(i,j).\]

We use the following asymptotic notation. If $G_n$ represents a graph of order $n$ in a sequence or family of graphs, and $f$ is some positive-valued graph invariant which may be expressed as a function of $n$, then we write $f(G_n) = O(g(n))$ if $\limsup_{n\to\infty} \frac{f(G_n)}{g(n)}$ is finite, and understand this to mean that $f(G_n)$ has at most the same order of magnitude as $g(n)$. Furthermore, we write $f(G_n) = \Omega(g(n))$ if $g(n) = O(f(G_n))$, and $f(G_n) = \Theta(g(n))$ if $f(G_n) = O(g(n))$ and $f(G_n) = \Omega(g(n))$. We write $f(G_n) = o(g(n))$ if $\lim \frac{f(G_n)}{g(n)} = 0$, $f(G_n) = \omega(g(n))$ if $\lim \frac{f(G_n)}{g(n)} = \infty$, and $f(G_n) \sim g(n)$ if $\lim \frac{f(G_n)}{g(n)} = 1$. We usually omit the subscript $n$ when referring to a family of graphs of order $n$.

Some of our proofs are probabilistic, and we shall use the
following notation for the random variables corresponding to 
the random walk on $G$ (or, in modified form, to the random walk on 
$\overline{G}$).
We write the random walk Markov chain on $G$ as a sequence 
$X_0,X_1,X_2,\ldots$ of $V(G)$-valued random variables, with one-step transition probabilities 
given by $P(\cdot,\cdot)$ as follows:
\begin{equation}
    \label{eq.Pvw}
    P(v,w)  \;=\;   \Pr(X_1\,=\,w\,|\,X_0=v)   \;=\;
     \begin{cases}    (d_v)^{-1},   & \text{if }\{v,w\}\in E(G); \\   0, & \text{otherwise} .  \end{cases}
\end{equation}
More generally, for $t\in \mathbb{N}$, we shall write the $t$-step transition probabilities  
as $P^{(t)}(v,w)\,=\,\Pr(X_t=w\,|\,X_0=v)$.  
Assuming that $G$ is connected, the Markov chain has a unique stationary distribution, which 
we shall denote $\pi(\cdot)$.  It is well known that 
\[    \pi(v)   \;=\;  \frac{ {d_v}}{2m}    \hspace{5mm}\hbox{for every $v\in V(G)$, where $m=|E(G)|$.}
\]
For the Markov chain that starts at a specified vertex $v$, we write
$P_v(\cdot)$ to denote $P(\cdot |X_0=v)$ and 
${\mathbb E}_v(\cdot)$ to denote the expectation 
$\mathbb{E}(\cdot| X_0=v)$.

For a given vertex $j$, 
we define the random variable $T_j$ to be the 
\textit{first passage time} of the vertex $j$ by the 
random walk $\{X_t\}$; that is,
\[     T_j \;=\;  \min\{  t\geq 1\,:  \;  X_t=j \}  \,.
\]
Thus our mean first-passage time $m_{i,j}$ is $\mathbb{E}_i(T_j)$.
When $j=i$, we sometimes call $m_{ii}$ the mean return time to $i$.
It is a well known fact about Markov chains (e.g., Proposition 1.14(ii) of Levin, Peres, and Wilmer 2009) that
\begin{equation}
    \label{eq.piprop}
      \mathbb{E}_i(T_i) \;=\;  \frac{1}{\pi(i)}   \;=\;   \frac{2m}{d_i}  
      \hspace{5mm}\hbox{for every $i\in V(G)$.}
\end{equation}

\section{Nordhaus-Gaddum when maximum degree is $n-O(1)$}\label{sec:n-O1}

Consider a random walk on a connected graph $G$ with $n$ vertices and $m$ edges. Let $m_{i,j}$ be the mean first passage time from $i$ to $j$. It is known from \cite{chandra1989electrical} that $m_{i,j}+m_{j,i} = 2mr_{i,j}$. Hence, for any $i \in \{1, 2, \ldots, n\}$,
\begin{align}\label{inequality1}
	\mathcal{K}(G) = \sum_{\substack{j=1\\j\neq i}}^n \bigl(\frac{d_j}{2m}\bigr) m_{i,j}< \sum_{\substack{j=1\\j\neq i}}^n \bigl(\frac{d_j}{2m}\bigr) (m_{i,j}+m_{j,i}) =  \sum_{\substack{j=1\\j\neq i}}^n d_j r_{i,j}
\end{align}
where $d_j$ is the degree of vertex $j$. It is found in the proof of \cite[Proposition 3]{palacios2010kirchhoff} that $r_{i,j}\leq \diam(G)$. We immediately obtain the following upper bound on Kemeny's constant:
\begin{align}\label{inequality2}
	\mathcal{K}(G) < 2m\diam(G).
\end{align}

\begin{remark}\label{remark: n^3}
For a graph $G$ of order $n$, the maximum diameter is $n-1$, and the maximum number of edges is $\frac{n(n-1)}{2}$, though these are not attained by the same graph. Thus this inequality \eqref{inequality2} gives an alternative proof that $\mathcal{K}(G)=O(n^3)$. Families of graphs with $\mathcal{K}(G)=\Theta(n^3)$ are given in \cite{breen2019computing}. 
\end{remark}

This inequality also gives insight into the influence of diameter on the value of Kemeny's constant. In particular, we have the following result for graphs of small diameter.
\begin{proposition}\label{prop:fixed diameter}
    Let $G$ be a graph on $n$ vertices with $\diam(G) = O(1)$. Then, $\mathcal{K}(G)=O(n^2)$.
\end{proposition}

The following example shows that a sequence of graphs of fixed constant diameter can attain $\mathcal{K}(G) = \Theta(n^2)$.
\begin{example}\label{ex:fixed diameter}
	Let $n\geq 4$. Consider a graph $G$ obtained from $K_{a}$ and $K_{b}$ by adding an edge between a vertex in $K_{a}$ and a vertex in $K_{b}$. The order of $G$ is $n=a+b$ and $\diam(G) = 3$. Applying \cite[Proposition 2.2]{breen2022kemeny}, and letting $m$ denote the number of edges in $G$, we have
	\begin{align*}
		\mathcal{K}(G) = & \frac{1}{m}\left(\frac{1}{2}(a-1)^3+\frac{1}{2}(b-1)^3+\frac{(b^2-b+2)(a-1)^2}{a}+\frac{(a^2-a+2)(b-1)^2}{b}\right)\\
		&+\frac{1}{2m}(a^2-a+1)(b^2-b+1).
	\end{align*}	
	Setting $a=\ceil{\frac{n}{2}}$ and $b = \lfloor\frac{n}{2}\rfloor$, we can find that $\mathcal{K}(G)\sim \frac{n^2}{8}$.
\end{example}

In the remainder of this section, we determine an upper bound on each summand $d_jr_{i,j}$ in \eqref{inequality1} in order to find an upper bound on $\mathcal{K}(G)$. We pursue this by considering spanning trees and spanning 2-forests of specific types. First, we introduce some notation. Given a connected graph $G$ and two vertices $i$ and $j$, we let $\tau_{i,j}$ denote the number of spanning trees of $G$ in which $i$ and $j$ are adjacent. We use $\mathcal{F}(x,y;z)$ to denote the set of spanning $2$-forests of $G$ where one contains vertices $x$ and $y$, and the other contains $z$.

\begin{lemma}\label{lem:forest i and j adjacent}
	Let $G$ be a connected graph on $n$ vertices, and suppose that $i$ and $j$ are adjacent vertices in $G$. Then $f_{i,j}=\tau_{i,j}$.
\end{lemma}
\begin{proof}
We prove this by giving a bijection between the two sets. Let $f\in \mathcal{F}(i;j)$. Then $i$ and $j$ are in different components of $f$, but they are adjacent in $G$. The addition of the edge $\{i,j\}$ to $f$ produces a spanning tree of $G$ in which $i$ and $j$ are adjacent; it is easily seen that this is a bijective map whose inverse constitutes removing the edge $\{i,j\}$ from a spanning tree of $G$ in which this edge appears, producing a unique forest in $\mathcal{F}(i;j)$.
\end{proof}

\begin{theorem}\label{thm:upper bound d f}
	Let $G$ be a connected graph on $n$ vertices, and $i,j\in V(G)$. Let $Z_j$ be the set of vertices $v$ in $N(j)$ such that $v\neq i$ and $v$ is not adjacent to $i$, and $Z_i$ be the set of vertices $v$ in $N(i)$ such that $v\neq j$ and $v$ is not adjacent to $j$. Then,
	\begin{align*}
		d_jf_{i,j} \leq \begin{cases*}
			2\tau + \sum_{v\in Z_j}\left|\mathcal{F}(j,v;i)\right|, & \text{if $i$ and $j$ are not adjacent,}\\
			2\tau + \sum_{v\in Z_j}\left|\mathcal{F}(j,v;i)\right|-\tau_{i,j}, & \text{if $i$ and $j$ are adjacent,}\\
		\end{cases*}
	\end{align*}
	where the equality holds if and only if for each $z \in Z_i$, any path from $z$ to $j$ in $G$ passes through $i$. Moreover, for any fixed $i$,
    \begin{align}\label{the upper for Kemeny}
        \mathcal{K}(G)< \sum_{j\neq i}d_jr_{i,j} < 2(n-1)+\sum_{j\neq i}|Z_j|r_{i,j}.
    \end{align}
\end{theorem}
\begin{proof}
	Let $i$ and $j$ be vertices of $G$. Let $N(j)\backslash\{i\}=\{v_1,\dots,v_{s_0}\}$. If $i$ and $j$ are adjacent, then $s_0=d_j-1$; otherwise, $s_0=d_j$. 
	For $k=1,\dots,s_0$, we define $\mathcal{F}_{v_k}(i;j)$ to be the set obtained from $\mathcal{F}(i;j)$ by labelling $v_k$ as a root. Similarly, for each $v\in Z_j$, we define $\mathcal{F}_{v}(j,v;i)$ to be the set obtained from $\mathcal{F}(j,v;i)$ by labelling $v$ as a root. We define $\mathcal{T}_i$ (resp. $\mathcal{T}_j$) to be the set of spanning rooted trees of $G$ with root $i$ (resp. with root $j$). Let $\mathcal{F} = \bigcup_{k=1}^{s_0}\mathcal{F}_{v_k}(i;j)$, and $\mathcal{T} = \mathcal{T}_i\cup\mathcal{T}_j\cup\left(\bigcup_{v\in Z_j}\mathcal{F}_v(j,v;i)\right)$. Note that $|\mathcal{F}| = d_jf_{i,j}$ if $i$ is adjacent to $j$, and $(d_j-1)f_{i,j}$ otherwise. For the proof, we shall define a map between these two combinatorial objects $\mathcal{F}$ and $\mathcal{T}$, show that it is injective, and give the conditions under which it is also surjective; this determines the inequality as stated, and also characterizes the equality case.
	
	Define a map $\psi:\mathcal{F}\rightarrow \mathcal{T}$ as follows: for each $f \in \mathcal{F}$,
	\begin{itemize}
		\item[(i)] If $f\in\mathcal{F}_{v_k}(i;j)$ for some $1\leq k\leq s_0$ so that $v_k$ is connected to $i$ in $f$, then $\psi(f)$ is a spanning rooted tree of $G$ obtained from $f$ by adding an edge $\{v_k,j\}$ and re-labelling $j$ as a root.
		\item[(ii)] If $f\in\mathcal{F}_{v_k}(i;j)$ for some $1\leq k\leq s_0$ so that $v_k$ is connected to $j$ in $f$ and $v_k\notin Z_j$, then $\psi(f)$ is a spanning rooted tree of $G$ obtained from $f$ by adding an edge $\{v_k,i\}$ and re-labelling $i$ as a root.
		\item[(iii)] If $f\in\mathcal{F}_{v_k}(i;j)$ for some $1\leq k\leq s_0$ so that $v_k$ is connected to $j$ in $f$ and $v_k\in Z_j$, then $\psi(f)=f$.
	\end{itemize}
	Since $f$ is a spanning $2$-forest, $v_k$ is connected to either $i$ or $j$. If $v_k\notin Z_j$, then $v_k$ is adjacent to $i$, so the addition of $\{v_k,i\}$ in (ii) is well-defined. Hence, $\psi$ is well-defined.
	
	We claim that $\psi$ is injective. Suppose that $\psi(f) = \psi(g)$ for some $f,g\in\mathcal{F}$. If $\psi(f)$ and $\psi(g)$ are spanning forests, then by (iii), we have $f=g$. Suppose that $\psi(f)$ and $\psi(g)$ are spanning trees. Since $\psi(f)$ and $\psi(g)$ have the same root, the newly inserted edges to obtain $\psi(f)$ and $\psi(g)$ are both incident to either $i$ or $j$. Furthermore, they must be on the unique path from $i$ to $j$ in $\psi(f)$ and $\psi(g)$. Hence, the newly inserted edges are the same. Therefore, $f$ and $g$ must have the same root and so $f = g$.
	
	Now we consider under what circumstance $\psi$ is surjective. Let $t\in \mathcal{T}$. Note $\mathcal{F}_v(j,v;i)\subseteq \mathcal{F}_v(j;i)$. If $t\in \bigcup_{v\in Z_j}\mathcal{F}_v(j,v;i)$, then $\psi(t) = t$. 
	
	Suppose that $i$ and $j$ are not adjacent. Let $t\in\mathcal{T}_i\cup\mathcal{T}_j$. There exists a vertex $w$ such that $w\neq i$, $w\neq j$, and it is adjacent to the root $r_0$ of $t$ and on the path from $i$ to $j$ (here either $r_0=i$ or $r_0=j$).  Deleting the edge $\{w,r_0\}$ from $t$, assigning $w$ a root, we obtain a spanning $2$-forest $f$ with root $w$. If $t\in \mathcal{T}_j$, then $w=v_k$ for some $1\leq k\leq s_0$ so that $f\in\mathcal{F}_{v_k}(i;j)$ and $\psi(f)=t$. Suppose $t\in\mathcal{T}_i$. Then, either $w\in N(j)$ or $w\in Z_i$. Obviously, if $w\in N(j)$ then $\psi(f)=t$. Note that if $w\in Z_i$ then $i$ is not on the path from $w$ to $j$ in $t$. Therefore, $\psi$ is surjective if and only if for each $z \in Z_i$, any path from $z$ to $j$ in $G$ passes through $i$.
	
	For the remaining case, we assume that $i$ and $j$ are adjacent. Note that $s_0=d_j-1$. Let $t\in\mathcal{T}_i\cup\mathcal{T}_j$. If $t$ does not have an edge $\{i,j\}$, then the same argument above can be applied. Moreover, from Lemma \ref{lem:forest i and j adjacent}, the number of spanning trees containing $\{i,j\}$ in each of $\mathcal{T}_i$ and $\mathcal{T}_j$ is $f_{i,j}=\tau_{i,j}$. Hence, we have $(s_0-1)f_{i,j}\leq 2\tau + \sum_{v\in Z_j}\left|\mathcal{F}(j,v;i)\right|-2\tau_{i,j}$. Therefore, we obtain the upper bounds for $d_jf_{i,j}$ with the the same condition for the equality.

    Finally, we can see $|\mathcal{F}(j,v;i)|\leq f_{i,j}$ and hence    
    $d_jf_{i,j}\leq 2\tau+|Z_j|f_{i,j}$ for any $j\in V(G)\backslash\{i\}$. Note that if $i$ and $j$ are adjacent, then $d_jf_{i,j}< 2\tau+|Z_j|f_{i,j}$. Equation \eqref{the upper for Kemeny} follows from \eqref{inequality1}.
\end{proof}


\begin{corollary}
	Let $G$ be a connected graph on $n$ vertices with maximum degree $n-1$. Then, $\mathcal{K}(G)<2(n-1)$ and hence $\mathcal{K}(G)=\Theta(n)$.
\end{corollary}

\begin{corollary}
	Let $G$ be a connected graph on $n$ vertices with maximum degree $n-O(1)$. Then, $\mathcal{K}(G)=O(n)$. This implies that $\mathcal{K}(G)\mathcal{K}(\overline{G})=O(n^4)$ if $\overline{G}$ is connected.
\end{corollary}
\begin{proof}
    Assume that there is a constant $C$ such that $\Delta(G)\geq n-C$.  Let $i$ be a vertex with the maximum degree. It follows that $|Z_j|\leq C-1$ for every $j\in V(G)\setminus\{i\}$ and $\diam(G)\leq C+1$. Since $r_{ij}\leq \diam(G)$, we have $\sum_{j\neq i}|Z_j|r_{ij}  \leq (n-1)(C-1)(C+1)$.  The conclusion follows.
\end{proof}

\begin{example}
    The complements of the barbell-type graphs in \cite{breen2019computing} have maximum degree $n-3$, and the complement of a tree has maximum degree $n-2$. Hence, their Kemeny's constants are $O(n)$ when they are connected. In Section~\ref{sec:app}, $\mathcal{K}(G)\mathcal{K}(\overline{G})$ will be explicitly considered.
\end{example}

The following corollary will be used in the next section.

\begin{corollary}
    \label{cor-deltaplusbound}
    Let $G$ be a connected graph on $n$ vertices such that 
    $\Delta(G)+\delta(G)\geq n$.  Then
    \[     \mathcal{K}(G) \;< \; \left(  \frac{2\delta(G)}{\delta(G)+\Delta(G)-n+1}\right) \,n \,.
    \]   
    In particular, $\mathcal{K}(G)$ must be less than $2n^2$.
\end{corollary}
\begin{proof}
    Assume $\Delta(G)+\delta(G)\geq n$. 
    Since $|Z_j|\leq n-1-\Delta(G)$ for all $j\in V(G)\setminus\{i\}$, we find that
    \begin{align*}
       \sum_{j\neq i}|Z_j|r_{i,j}  \; \leq \;  (n-1-\Delta(G))\sum_{j\neq i}r_{i,j} \;\leq \;  (n-1-\Delta(G))\sum_{j\neq i}\frac{d_j \,r_{i,j}}{\delta(G)} .
    \end{align*}
From \eqref{the upper for Kemeny}, we have $\sum_{j\neq i}d_jr_{i,j} < 2(n-1)+\sum_{j\neq i}|Z_j|r_{i,j}$. Hence,
\[
   \left(\frac{\delta(G)+\Delta(G)-n+1}{\delta(G)}\right)  \sum_{j\neq i}d_jr_{i,j}  < 2n-2.
\]
Note that $\delta(G)+\Delta(G)-n+1>0$ by our assumption.  Using \eqref{the upper for Kemeny}, we obtain
\[   \mathcal{K}(G)  \;<\;  \frac{2\delta(G) \,n}{\delta(G)+\Delta(G)-n+1}.\qedhere
\]
\end{proof}

\section{Nordhaus-Gaddum bounds when maximum degree is $n-\Omega(n)$}
\label{sec:n-on}

In this section, we examine the relationship between some properties
of a graph $G$ and its complement   $\overline{G}$.
We begin by stating our results and presenting some short proofs.
Then we shall introduce some notation and background that we 
shall need for the longer proof of Proposition \ref{prop.complementLU}.

For $v\in V(G)$, in addition to our usual notation of $d_v$ for the degree of $v$ in $G$, we shall also write $\overline{d}_v$ for the degree of
$v$ in $\overline{G}$.  
Then $d_v+{\overline{d}}_v=n-1$ for every $v$.
Recall that $\Delta(G)$ is the maximum degree in $G$ and $\delta(G)$ is the minimum.

\begin{theorem} 
   \label{thm.complementU}
Let $U$ be a real constant such that $0<U<1$.  Then there is a constant    
$\Psi_U$ such that for every $n\in\mathbb{N}$ and every
graph $G$ on $n$ vertices such that $\Delta(G)\leq Un$, 
\[    \min\left\{  \mathcal{K}(G), \, \mathcal{K}(\overline{G})\right\}   \;\leq \; n\Psi_U   \,.
\]
\end{theorem}

\begin{remark}
\label{rem.Gconn}
  We do not assume that the graphs are connected.  Note that $G$ and $\overline{G}$
cannot both be disconnected.  In particular, for every graph $G$, at least one of $\mathcal{K}(G)$ or
$\mathcal{K}(\overline{G})$ is finite.
Consequently, it suffices to prove the theorem for all $n\geq n_0$, where $n_0$ is a $U$-dependent 
constant.
\end{remark}

Since $\Delta(\overline{G}) = n-1-\delta(G)$, the following corollary is immediate.

\begin{corollary}
   \label{cor.complementL}
Let $L$ be a real constant such that $0<L<1$.  Then    
for every $n\in\mathbb{N}$ and every
graph $G$ on $n$ vertices such that $\delta(G)\geq Ln$, 
\[    \min\left\{  \mathcal{K}(G), \, \mathcal{K}(\overline{G})\right\}   \;\leq \; n\Psi_{1-L}   \,,
\]
where $\Psi_{1-L}$ is the constant from Theorem \ref{thm.complementU} with $U=1-L$.
\end{corollary}

We also obtain a linear bound that holds for all regular graphs uniformly in degree.

\begin{corollary}
    \label{cor.regular}
There exists a constant $\Psi_{reg}$ such that 
\[    \min\left\{  \mathcal{K}(G), \, \mathcal{K}(\overline{G})\right\}   \;\leq \; n\Psi_{reg} 
\]
for every regular graph $G$ with $n$ vertices.
\end{corollary}

\smallskip
\noindent
Indeed, we can take $\Psi_{reg}$ to be $\Psi_{1/2}$ from Theorem \ref{thm.complementU}.
This is because for any regular $n$-vertex graph, either the graph or its complement must have its
degree less than or equal to $(1/2)n$.

\bigskip

As we shall show, Theorem \ref{thm.complementU} follows easily from 
the following proposition and Corollary \ref{cor-deltaplusbound}.

\begin{proposition} 
   \label{prop.complementLU}
Let $L$ and $U$ be real constants such that $0<L< U<1$.  Then there is a constant    
$\Psi_{L,U}$ such that for every $n\in\mathbb{N}$ and every
graph $G$ on $n$ vertices such that $Ln\leq \delta(G)\leq \Delta(G)\leq Un$, 
\[    \min\left\{  \mathcal{K}(G), \, \mathcal{K}(\overline{G})\right\}   \;\leq \; n\,\Psi_{L,U} \,.
\]
\end{proposition}


\bigskip
\noindent
\textbf{Proof of Theorem \ref{thm.complementU}:}
First, recalling Remark \ref{rem.Gconn}, we note that it suffices to prove the result for sufficiently large $n$.
So we shall assume that $n\geq 12/(1-U)$.

Without loss of generality, assume $U>1/2$.  
Let $L=(1-U)/4$.

Let $G$ be a graph on $n$ vertices such that $\Delta(G)\leq Un$.
On the one hand, if $\delta(G)\geq Ln$, then Proposition \ref{prop.complementLU} implies
that $\min\left\{  \mathcal{K}(G), \, \mathcal{K}(\overline{G})\right\}   \;\leq \; n\Psi_{L,U}$.
On the other hand, if $\delta(G)<Ln$, then consider the complement:
\begin{align*}
   \Delta(\overline{G})+\delta(\overline{G})  \;& = \;  (n-1-\delta(G))  \,+\,(n-1-\Delta(G))
   \\
    & > \;  2n-2 -Ln-Un
    \\
   & \geq \;  2n +1 -\left(\frac{1-U}{4}\right)n   -\left(\frac{1-U}{4}\right)n   -Un  
     \\
   &  \hspace{45mm}
       \left(\text{since $3\leq\frac{(1-U)n}{4}$}\right)
   \\
   & =\;n+ 1+ \left( \frac{1-U}{2}\right)n \,.
\end{align*}
Now Corollary \ref{cor-deltaplusbound}
tells us that 
\[   \mathcal{K}(\overline{G}) \;\leq \; \left(  \frac{2n}{\left(\frac{1-U}{2}\right)n } \right) n  \;=\;  \frac{4n}{1-U} \,.  
\]
Theorem  \ref{thm.complementU} follows. 
\hfill $\Box$

\bigskip

We now introduce some notation for this section.

We shall write $V$ for $V(G)$, which is the same as $V(\overline{G})$.
However, for edges, we shall be careful to distinguish between
$E(G)$ and $E(\overline{G})$.

When $S$ and $T$ are disjoint subsets of $V$, we define $[S,T]_G$ to be the set of all edges of $G$
that have one endpoint in $S$ and one endpoint in $T$. 
The analogue in $\overline{G}$ is $[S,T]_{\overline{G}}$.

Recall from Equation (\ref{eq:eig}) that
for every graph $G$ with $n$ vertices, we have
\begin{equation}
   \label{eq.Keigensum}
   \mathcal{K}(G)  \;=\;  \sum_{i=2}^n  \frac{1}{1-\lambda_i}
\end{equation}
where $\lambda_1=1\geq \lambda_2\geq \lambda_3 \geq \ldots \geq \lambda_n$ are the eigenvalues of the 
normalized adjacency matrix $D^{-1}A$ (i.e., the transition probability matrix of the random walk on $G$).  
(Equation (\ref{eq.Keigensum})  also holds if $G$ is disconnected, on the understanding that for a disconnected graph $G$
we have $\mathcal{K}(G)=\infty$ and  $\lambda_2=1$.) 
Therefore,  we have
\begin{equation}
   \label{eq.Kembounds1}
      \frac{1}{1-\lambda_2}  \;\leq \;\mathcal{K}(G_n)   \;\leq \;  \frac{n}{1-\lambda_2} 
       \hspace{5mm}\text{for every  graph $G_n$ on $n$ vertices}.
\end{equation}

For $S\subseteq V$, let 
\[    {\rm vol}(S)\;:=\;  \sum_{v\in S} d_v \,.
\]
Note that ${\rm vol}(V)\,=\,2|E(G)|$.  For example, in a regular
graph of degree $d$, ${\rm vol}(S)=d|S|$.
The \textit{bottleneck ratio} of the graph $G$ is  defined to be
\begin{equation}
   \label{eq.defPhi}
      \Phi  \;=\;  \Phi(G)  \;=\;  \min_{S\subseteq V:  \,0<{\rm vol}(S)\leq |E(G)|} \frac{  |\,[S,S^c]_G|}{{\rm vol}(S)}  \,.
\end{equation}
The classic work of Jerrum and Sinclair \cite{JeSi} and Lawler and Sokal \cite{LaSo} proved 
\begin{equation}
   \label{eq.jerrum}
       \frac{\Phi^2}{2}  \;\leq \;  1-\lambda_2  \;\leq \;  2\Phi.
\end{equation}
(Again, we note consistency when $G$ is disconnected, since then $\Phi=0$.)
For an overview from the point of view of discrete reversible Markov chains, of which our
context is a special case,
see Sections 7.2 and  13.3 of \cite{LPW}.

Combining (\ref{eq.Keigensum}) and (\ref{eq.jerrum}) yields
\begin{equation}
    \label{eq.Kemjerrbd}
      \frac{1}{2\Phi(G)}   \;\leq \;  \mathcal{K}(G)   \;\leq \;  \frac{2n}{\Phi(G)^2}  \,.
\end{equation}

\bigskip
\noindent
\textbf{Proof of Proposition \ref{prop.complementLU}:}
Let $L$ and $U$ be the constants specified in the statement of the proposition.

We shall require a large parameter $M$ and a small parameter $\epsilon$ which we define by
\begin{equation}
   \label{eq.defMeps}
     M \; :=\;  \frac{8}{L(1-U)}    \hspace{5mm} \text{and} \hspace{5mm}
         \epsilon \; := \;  \frac{1}{M^2}  \;=\;   \frac{L^2(1-U)^2}{64} \,.
\end{equation}
We also let 
\begin{equation}
     \label{eq.psiminusdef}
         \Psi_-  \; :=\;  \frac{2}{\epsilon^2}  \,.
\end{equation}
We shall obtain another constant $\Psi_+$ with the property 
that, for the graphs under consideration,
\begin{verse}
     If $\mathcal{K}(G)   \,>\, n\Psi_-$, then $\mathcal{K}(\overline{G}) \,\leq\, n\Psi_+$.
\end{verse}
Then the proposition will follow by taking $\Psi_{LU}$ to be $\max\{\Psi_-,\Psi_+\}$.

Let $G$ be a graph with $n$ vertices such that 
$\mathcal{K}(G)>n\Psi_-$ and $Ln\leq \delta(G)\leq \Delta(G)\leq Un$.  Then
\begin{equation}
    \label{eq.degbound1}   
     Ln   \;\leq \;  d_v  \; \leq \;      Un  \hspace{10mm} \forall v\in V \,.
\end{equation}
For future reference, we also note that
\begin{equation}
   \label{eq.volSvol}
     \frac{ {\rm vol}(T)}{Un}  \;\leq \;  |T|  \;\leq \;  \frac{{\rm vol}(T)}{Ln}   \hspace{10mm} \forall \,T\subseteq V.
\end{equation} 
In the complement $\overline{G}$, we have $n-1-Un \leq \delta(\overline{G})\leq \Delta(\overline{G})\leq n-1-Ln  < (1-L)n$.
By restricting our attention to sufficiently large $n$ (it suffices that $n>2/(1-U)$), we have that
\begin{equation}
    \label{eq.degbound1c}   
     \frac{(1-U)\,n}{2}   \;\leq \;  \overline{d}_v  \; \leq \;      (1-L)\,n  \hspace{10mm} \forall \,v\in V \,.
\end{equation}
Since it suffices to prove the proposition for sufficiently large $n$, we shall henceforth assume
that $n$ is large enough so that Equation (\ref{eq.degbound1c}) holds.

Here is an outline of the proof.

\smallskip
\noindent
\textit{Step 1:}
Since $\mathcal{K}(G)$ is large, the bottleneck ratio of $G$ must be small.  This means
that there is a set $S$ of vertices such that there are relatively few edges in the cut $[S,S^c]_G$.
Moreover, we show that the sets $S$ and $S^c$ each have at least order $n$ vertices.
The idea then is that the corresponding complementary cut 
$[S,S^c]_{\overline{G}}$ contains most of the $|S|\,|S^c|$ possible edges between 
$S$ and $S^c$.  (If it contained every possible edge between $S$ and $S^c$, which happens only if
$G$ is disconnected, then Theorem \ref{thm.joinbound} would imply that 
$\mathcal{K}(\overline{G})$ would be $O(n)$.)

\smallskip
\noindent
\textit{Step 2:}  We specify two sets of vertices $A\subseteq S$ and $B\subseteq S^c$ such that every vertex in $A$
is connected (in $\overline{G}$) to most of $S^c$, every vertex in $B$ is connected to most of $S$, and
the set differences $S-A$ and $S^c-B$ are both relatively small.

\smallskip
\noindent
\textit{Step 3:}  Consider an arbitrary vertex $x$, whose  first passage
time in $\overline{G}$ we shall be investigating.
We show first that either $A$ or $B$ contains a substantial number (order $n$) of neighbours (in $\overline{G}$) of $x$.

\smallskip
\noindent   
\textit{Step 4:}  We now look at properties of the random walk on $\overline{G}$.  
We prove that there is a $\Theta>0$ (depending on $L$ and $U$ but not $n$) such that 
from every vertex of $A$ (respectively, $B$), the probability of entering $B$ (respectively, $A$) on the 
next step is at least $\Theta$.  We also prove that from any vertex, the probability that the next 
step is to $A\cup B$ is not small (in fact, at least $1/2$).

\smallskip
\noindent  
\textit{Step 5:}
Wherever the walker happens to be, there is a probability of at least order $1/n$ that vertex $x$ will be visited
within the next four steps (it is always likely to enter $A\cup B$ in one step, then it has a reasonable 
chance of visiting a neighbour of $x$ in either one or two more steps, and from that neighbour
the probability of going to $x$ is at least $1/n$).  
So the first passage time of $x$ is at least as fast as the time until the first Heads when we toss
a coin every four time steps and the probability of Heads is of order $1/n$ on each toss.
The expected time to the first Heads in this situation is order $n$, with a constant that is
independent of $x$.  The bound $\mathcal{K}(\overline{G})=O(n)$ follows.

\medskip
Now we return to the proof.

\medskip
\textit{Step 1.}
Since $\mathcal{K}(G)  \,>\, n\Psi_-$, it follows from Equations (\ref{eq.Kemjerrbd}) 
and (\ref{eq.psiminusdef}) that
$\Phi(G) \; <  \;   \epsilon$.   
Therefore, by the definition of $\Phi(G)$ in Equation (\ref{eq.defPhi}), there is a nonempty subset $S$ of $V$
such that ${\rm vol}(S)\leq |E(G)|$ and 
\begin{equation}
    \label{eq.SPhibeta}
       |\,[S,S^c]_G|   \;< \;  \epsilon \,{\rm{vol}}(S) \;\leq \;  \epsilon \,n\,U\,|S| 
\end{equation}
(the second inequality uses Equation (\ref{eq.volSvol})).  Also, we have
\begin{equation*}
   |S^c|\,Un   \;\geq \;  {\rm vol}(S^c)   \;=\; 
   2|E(G)|-{\rm vol}(S)  \;\geq \;|E(G)| \;\geq \; \frac{L\,n^2}{2},
\end{equation*}
and hence
\begin{equation}
    \label{eq.2UScn}
      \frac{2U}{L}\, |S^c| \;\geq \; n \,.
\end{equation}
    
 Let $C\,=\,  (L-\epsilon U)/2$.   
 Since $\epsilon<L$ by  Equation (\ref{eq.defMeps}), we see that
 \begin{equation}
    \label{eq.Cbound}
    \frac{L}{2} \;>\;   C  \; >  \;  \frac{L(1-U)}{2}    \;> \; 0.
\end{equation}
We claim that $|S|\,>\,n C$.   If not, then every vertex $v$ in $S$ has at most 
$nC$ neighbours in $S$ (in $G$), and hence 
\[
   |\,[v,S^c]_G| \;  \geq    \; Ln \,-\, Cn \;> \;   \epsilon U n \,.  
\]
Summing the above inequality over all $v\in S$ gives $|\,[S,S^c]_G|  \,> \, n\,U\epsilon \,|S|$, which 
contradicts Equation (\ref{eq.SPhibeta}).  This proves the claim that
\begin{equation}
   \label{eq.SgtnC}
      |S| \;>\;  nC  \,  \hspace{5mm}\text{where }C\;=\;  \frac{L-\epsilon U}{2}.
\end{equation}


\medskip

\textit{Step 2.}
Note that $M\epsilon \,=\, L(1-U)/8\,<\,1/2$ and hence
\begin{equation}
   \label{eq.Mepsbd}
       1-M\epsilon   \;>  \;  \frac{1}{2}  \,.
\end{equation}
Define the sets of vertices $A\subseteq S$ and $B\subseteq S^c$ by
\begin{eqnarray}
  \label{eq.Adef1}
  A  & = &  \{v\in S\,:\, |\,[v,S^c]_{\overline{G}}| \,\geq \, |S^c|(1-M\epsilon)    \,\}
     \\
     \label{eq.Adef2}
    & = &  \{v\in S\,:\, |\,[v,S^c]_{G}| \,\leq\, M\epsilon \, |S^c|   \, \} \, , \hspace{7mm} \text{and}
     \\
     \label{eq.Bdef1}
     B  & = &  \{w\in S^c\,:\, |\,[w,S]_{\overline{G}}| \,\geq \, |S|(1-M\epsilon)    \,\}
     \\
     \label{eq.Bdef2}
    & = &  \{w\in S^c\,:\, |\,[w,S]_{G}| \,\leq\, M\epsilon \, |S|   \, \}   \,.
\end{eqnarray}
From Equation (\ref{eq.Adef2}), it follows that 
\[       |\,[S,S^c]_G| \;\geq \;  \sum_{v\in S-A}|\,[v,S^c]_G|    \;\;\geq \;\;  |S-A|\times M\epsilon |S^c|
\]
and hence (using Equations (\ref{eq.SPhibeta}) and (\ref{eq.volSvol}), 
and the bounds ${\rm vol}(S)\leq |E(G)|\leq {\rm vol}(S^c)$)  that 
\begin{equation}
   \label{eq.SminusAbound}
     |S-A|\;\leq \;   \frac{  |\,[S,S^c]_G|}{M\epsilon \,|S^c|}  \;\leq \; \frac{\epsilon \, {\rm vol}(S)}{M\epsilon \,({\rm vol}(S^c)/Un)}  
         \;\leq \;  \frac{nU}{M}\,.
\end{equation} 
Similarly, we have     
\[       |\,[S,S^c]_G| \;\geq \;  \sum_{w\in S^c-B}|\,[w,S]_G|    \;\;\geq \; \;  |S^c-B|\times M\epsilon |S|
\]
and hence (using Equation (\ref{eq.SPhibeta}))   that 
\begin{equation}
   \label{eq.ScminusBbound}
     |S^c-B|\;\leq \;   \frac{  |\,[S,S^c]_G|}{M\epsilon \,|S|}  \;\leq \;
      \frac{nU}{M}   \,.
\end{equation} 
By Equations  (\ref{eq.SminusAbound}), (\ref{eq.SgtnC}),  (\ref{eq.ScminusBbound}), and (\ref{eq.2UScn}),  we have
\begin{eqnarray}
  \nonumber 
    |A|  &\geq &  |S|\,-\, n\frac{U}{M}  \;\; \geq \; n\left( C -  \frac{U}{M}  \right)   \quad\text{ and}
    \\
      \nonumber 
    |B|  &\geq &  |S^c|\,-\, n\frac{U}{M}  \;\; \geq \; n\left(  \frac{L}{2U} -  \frac{U}{M}  \right) \,.  
\end{eqnarray}
We note the above two lower bounds are strictly positive because
\begin{equation}
  \label{eq.U2MLCbound}
   \frac{U}{M}\;=\;  \frac{UL(1-U)}{8}   \;<\;  \frac{C}{4} \;<\; \frac{L}{8U} \hspace{6mm}\text{by Eq.\ (\ref{eq.Cbound})}.
\end{equation}

\medskip

\textit{Step 3.}
Now consider an arbitrary vertex $x\in V$.  We want to show that the expected time for a random walk
on $\overline{G}$ to reach $x$ from 
any other initial vertex $v$ is bounded by $n\Psi_+$ uniformly in $x$ and $v$.

We now introduce some notation.  
For a vertex $w\in V$, as $N(w)$ is the neighbourhood of $w$ in $G$, let $\overline{N}(w)$ be the neighbourhood of $w$
in $\overline{G}$:
\begin{eqnarray*}
    N(w)   & =&   \{  u\in V\,:\,  \{w,u\}\in E(G)\,\}    \hspace{5mm}\text{and}
    \\
      \overline{N}(w)   & = &   \{  u\in V\,:\,  \{w,u\}\in E(\,\overline{G}\,)\,\}\,.
\end{eqnarray*}
Next, we define the set of neighbours in $\overline{G}$  of $x$ in $A$ and in $B$ respectively: 
\[
     A_x  \;=\;  A\cap \overline{N}(x)    \hspace{5mm}\text{and}\hspace{5mm}  B_x\;=\;  B\cap \overline{N}(x) \,,
\]
Then we have
\[  |A_x|\,+\,|B_x|  \;\leq \;  \overline{d}_x  \;\leq \;  |A_x|\,+\,|B_x|\,+\,|S-A|  \,+\,|S^c-B|\,.
\]
Therefore, by Equations (\ref{eq.degbound1c}), (\ref{eq.SminusAbound}),  and (\ref{eq.ScminusBbound}),    
\begin{eqnarray}
  \nonumber
       |A_x|\,+\,|B_x|  & \geq &   \overline{d}_x  \,-\,  |S-A|  \,-\,  |S^c-B|
       \\
 \nonumber  
       & \geq & n\, \frac{1-U}{2}  \,-\,  2n\,\frac{U}{M}
          \\
                    \label{eq.AxBxbd2}  
          &  = &   2n\Gamma \,,  
          \\  
           \nonumber
          & &  \hspace{5mm} \text{where}\hspace{5mm}
     \Gamma \;:=\;  \frac{1}{2}     \left(\frac{1-U}{2} -\frac{2U}{M} \right)  \,.    
\end{eqnarray}
Since
\[   \frac{2U}{M}  \;=\;  \frac{UL(1-U)}{4}  \;<\;   \frac{1-U}{4} ,  
\]
it follows that 
\begin{equation}
   \label{eq.Gammabound}
        \Gamma \;>\;  \frac{1-U}{8} \,.
\end{equation}
We see immediately from Equation (\ref{eq.AxBxbd2}) that
\begin{equation}
  \label{eq.AxorBx}
   \text{At least one of  $|A_x|$ or $|B_x|$ is greater than or equal to $n\Gamma$.}
\end{equation}

\medskip

\textit{Step 4.}
To describe the random walk on the complement,
we shall slightly modify the Markov chain notation
introduced in Section \ref{sec:notation}.
The random walk process on $\overline{G}$ is a sequence 
$X_0,X_1,X_2,\ldots$ of $V$-valued random variables, with one-step transition probabilities 
given by 
\[
     \overline{P}(v,w)  \;=\;   \Pr(X_1\,=\,w\,|\,X_0=v)   \;=\;
     \begin{cases}    \left(\overline{d}_v\right)^{-1}  
     & \text{if }\{v,w\}\in E(\overline{G}) 
     \\   0 & \text{otherwise} .  \end{cases}
\]
For $t\in \mathbb{N}$,  the $t$-step transition probabilities  
are $\overline{P}^{(t)}(v,w)\,=\,\Pr(X_t=w\,|\,X_0=v)$.

For $w\in B$, we have
\begin{eqnarray}
     \nonumber
     \overline{P}(w,A)    & = &   |\,[w,A]_{\overline{G}}| \, / \, \overline{d}_w
      \\
      \nonumber
      & \geq & \frac{ |\,[w,S]_{\overline{G}}| \,-\, |S-A| }{n(1-L)}  
          \hspace{15mm}\text{(by Eq.\  (\ref{eq.degbound1c}))}      
  \\
      \nonumber
        & \geq & \frac{ |S|(1-M\epsilon) \,-\,n\frac{U}{M}  }{n(1-L)}   
           \hspace{12mm}\text{(by  Eqs.\ (\ref{eq.Bdef1}) and  (\ref{eq.SminusAbound}))}
           \\
      \nonumber
        & \geq & \frac{ nC(1-M\epsilon) \,-\,n\frac{U}{M}  }{n(1-L)}   
            \hspace{12mm}\text{(by  Eq.\  (\ref{eq.SgtnC}))}    
         \\
      \label{eq.PwAbound}
        & = & \frac{ C(1-M\epsilon) \,-\,\frac{U}{M}  }{1-L} \quad  =:\;  \Theta. 
\end{eqnarray}
We know that $\Theta>0$ from  Equations (\ref{eq.Mepsbd}) and (\ref{eq.U2MLCbound}).  
Similarly, for $v\in A$,
\begin{eqnarray}
     \nonumber
     \overline{P}(v,B)    & = &   |\,[v,B]_{\overline{G}}| \, / \, \overline{d}_v
      \\
      \nonumber
      & \geq & \frac{ |\,[v,S^c]_{\overline{G}}| \,-\, |S^c-B| }{n(1-L)}   
         \hspace{15mm}\text{(by Eq.\  (\ref{eq.degbound1c}))}
      \\
      \nonumber
        & \geq & \frac{ |S^c|(1-M\epsilon) \,-\,n\frac{U}{M}  }{n(1-L)}   
           \hspace{15mm}\text{(by Eqs.\ (\ref{eq.Adef1}) and  (\ref{eq.ScminusBbound}))}
           \\
      \nonumber
        & \geq & \frac{ \frac{nL}{2U}(1-M\epsilon) \,-\,n\frac{U}{M}  }{n(1-L)}   
          \hspace{15mm}\text{(by Eq.\ (\ref{eq.2UScn}))}         
         \\
     \nonumber
        & \geq  & \frac{ C(1-M\epsilon) \,-\,\frac{U}{M}  }{1-L}   \hspace{15mm}\text{(by Eq.\ (\ref{eq.U2MLCbound}))}
        \\
           \label{eq.PvBbound}
        & = &  \; \Theta.
\end{eqnarray}

For every $z\in V$,   
\begin{eqnarray}
   \nonumber
      \overline{P}(z,(A\cup B)^c)   & \leq &   (|S-A|\,+\,|S^c-B|) \,/\, \overline{d}_z
      \\
      \nonumber
          & \leq & \frac{ 2nU/M }{n(1-U)/2} 
             \hspace{10mm}\text{(by Eqs.\ (\ref{eq.SminusAbound}), (\ref{eq.ScminusBbound}),
              and (\ref{eq.degbound1c}))}
      \\
   \nonumber   
    & = &   \frac{4U}{M(1-U)} \,.
\end{eqnarray}
Therefore
\begin{equation}
    \label{eq.PzABbound}
    \overline{P}(z, A\cup B)   \;\geq \;  \Upsilon \hspace{5mm}
    \text{for every vertex $z$, where}\hspace{5mm}
  \Upsilon\,:= \,  1\,-\,    \frac{4U}{M(1-U)}\,. 
\end{equation}
Note that $\Upsilon>0$ because $\frac{4U}{M(1-U)} \,=\, \frac{UL}{2}\,<\,\frac{1}{2}$.

\medskip

\textit{Step 5.}
Now we consider the two cases in Equation (\ref{eq.AxorBx}).
\begin{verse}
    \underline{Case 1}:  $|A_x| \;\geq \; n\Gamma$;  
    \\
    \underline{Case 2}:  $|B_x| \;\geq \, n\Gamma$.
\end{verse}

First we assume that Case 1 holds, i.e.\ that $|A_x|\geq n\Gamma$.
(The proof for Case 2 will be very similar.)
In this case,  for $w\in B$ we have
\begin{align*}
     \{ t\in A_x : t\not\in \overline{N}(w) \} & \;\subseteq \;   \{t\in S: t\not\in \overline{N}(w)\}
     \\
     &\; = \; \{ t\in S:  t\in N(w)\}
       \\
       &\; = \;  [w,S]_G \,,
\end{align*}
and hence, using Equation  (\ref{eq.Bdef2}),
\[      |[w,A_x]_{\overline{G}}|   \;\geq \;   |A_x|\,-\,  |[w,S]_G|   \;\geq \; n\Gamma-M\epsilon |S| \;\geq \; n\Gamma \,-\ nM\epsilon.
\]
Therefore, for all $w\in B$, we have from Equations    (\ref{eq.defMeps}) and   (\ref{eq.degbound1c}) that
\begin{align}
      \overline{P}(w,A_x) \; & \geq  \;     \frac{n\Gamma-nM\epsilon}{ \overline{d}_w}
      \nonumber
      \\
      \nonumber
      & \geq \;   \frac{\Gamma - \frac{L(1-U)}{8} }{1-L}
      \\
      \nonumber
      & \geq \;  \frac{\Gamma - L\Gamma}{1-L}   \hspace{5mm}\text{(by Eq.\ (\ref{eq.Gammabound}))}
      \\
      \label{eq.PwAxbound}
      & = \;  \Gamma \,.
\end{align}

Since $\overline{P}(u,x)\geq 1/n$ for every $u$ in $A_x$, we see that for every $w\in B$ we have
\begin{align}
     \nonumber
        \overline{P}^{(2)}(w,x)  \; & \geq \;   \sum_{u\in A_x}  \overline{P}(w,u)\,\overline{P}(u,x)
        \\
        \nonumber
        & \geq \;   \sum_{u\in A_x}  \overline{P}(w,u)\,\frac{1}{n}    
        \\
        \nonumber
        & = \; \overline{P}(w,A_x)  \,\frac{1}{n}     
        \\
        \label{eq.P2xbound} 
        & \geq \; \frac{\Gamma}{n}        \hspace{12mm}\text{(by Equation (\ref{eq.PwAxbound}))}.  
\end{align}            
Next, for $v\in A$, we have 
\begin{eqnarray}
            \nonumber
        \overline{P}^{(3)}(v,x)  & \geq &   \sum_{w\in B}  \overline{P}(v,w)\,\overline{P}^{(2)}(w,x)
        \\
        \nonumber
        & \geq &   \sum_{w\in B}  \overline{P}(v,w)\,\frac{\Gamma}{n}      \hspace{11mm}\text{(by Equation (\ref{eq.P2xbound}))}   
        \\
        \nonumber
        & = & \overline{P}(v,B)  \,\frac{\Gamma}{n}  \,.      
\end{eqnarray}                       
Hence it follows from
Equation (\ref{eq.PvBbound}) that
\begin{equation}
   \label{eq.P3Axbound}
      \overline{P}^{(3)}(v,x)  \;\geq \;  \frac{\Theta\,\Gamma}{n}   \hspace{7mm}\text{for all }v\in A.   
\end{equation}     

The first passage time of the vertex $x$ by the random walk on 
$\overline{G}$ is the random variable $T_x$
defined by
\[     T_x \;=\;  \min\{  t\geq 1\,:  \;  X_t=x \}  \,.
\]
Then Equations (\ref{eq.P2xbound}) and (\ref{eq.P3Axbound}) imply that
\begin{equation}
    \label{eq.tauxAB}
    \overline{P}(T_x\leq 3 \,|\, X_0=w)  \;\geq \;  \frac{\Theta\,\Gamma}{n}    \hspace{7mm}\text{for all }w\in A\cup B.  
\end{equation}
Next, for every $s\in V$, we have
\begin{eqnarray}
   \nonumber
   \overline{P}(T_x \leq 4  \,|\,X_0=s)  & \geq &
        \sum_{w\in A\cup B}\overline{P}(s,w)\,\overline{P}(T_x\leq 3\,|\,X_0=w)
      \\
      \nonumber
      & \geq &   \overline{P}(s,A\cup B) \, \frac{\Theta\,\Gamma}{n}    \hspace{11mm}\text{(by Equation (\ref{eq.tauxAB}))}   
      \\
      \label{eq.tau4bound}
      & \geq & \frac{\Upsilon \,\Theta\,\Gamma}{n}      \hspace{11mm}\text{(by Equation (\ref{eq.PzABbound}))}.  
\end{eqnarray}
Therefore
\[  
     \overline{P}(T_x >4   \,|\,X_0=s)   \;\leq \;  1\,-\, \frac{\Upsilon\,\Theta\,\Gamma}{n}   \hspace{6mm}\text{for every }s\in V,  
\]
and consequently, for every $s\in V$ and $j> 0$, 
\begin{align*}   
     \overline{P}(T_x  >4 +j  \,|\,X_0=s )
      \; & = \; 
        \sum_{y\in V, \,y\neq x}   \overline{P}(T_x >4 +j  \,|\,X_0=s,\,X_j=y, \,T_x>j)  \overline{P}(X_j=y, \,T_x >j  \,|\,X_0=s)  
        \\
        & = \;    \sum_{y\in V, \,y\neq x}   \overline{P}(T_x >4  \,|\,X_0=y)   \,      \overline{P}(X_j=y, \,T_x >j  \,|\,X_0=s)  
        \\
        & \hspace{80mm}\text{(by the Markov property)}
       \\
       & \leq \;  \left(1\,-\, \frac{\Upsilon\,\Theta\,\Gamma}{n} \right)  \,    \overline{P}(T_x >j  \,|\,X_0=s)   \,.           
\end{align*}
By induction on $k\in \mathbb{Z}^+$, we obtain
\begin{equation}
   \label{eq.Ptau4k}
     \overline{P}(T_x >4k  \,|\,X_0=s)   \;  \leq \;  \left(1\,-\, \frac{\Upsilon\,\Theta\,\Gamma}{n}\right)^k  \hspace{6mm}\text{for every }s\in V,\,k\geq 0.
\end{equation}      
Using that fact that for any nonnegative random variable $Z$
\[     \mathbb{E}(Z)   \;\leq \;  \mathbb{E}(\lceil Z \rceil)  \;=\;  \sum_{k=0}^{\infty}  \Pr(Z>k)  \,,
\]
we obtain (writing $\overline{\mathbb{E}}$ for the expectation with respect to $\overline{P}$)
\begin{eqnarray}
  \nonumber
  \overline{\mathbb{E}}\left( \left. \frac{T_x}{4}\,\right| \,X_0=s\right)   & \leq  & 
    \sum_{k=0}^{\infty}  \overline{P}\left( \left. \frac{T_x}{4} > k\,\right|\,X_0=s\right)
  \\
  \nonumber
  & \leq &   \sum_{k=0}^{\infty}  \left(1\,-\, \frac{\Upsilon\,\Theta\,\Gamma}{n}\right)^k    
     \hspace{7mm}\text{(by Eq.\ (\ref{eq.Ptau4k}))}
  \\
    \nonumber 
  & = &   \frac{n}{\Upsilon\,\Theta\,\Gamma}  \hspace{12mm}\text{for every }s\in V. 
\end{eqnarray}   
To summarize:
\begin{equation}
  \nonumber  
   \text{If Case 1 holds, then \; $\overline{\mathbb{E}}(T_x\,|\,X_0=s)  \;\leq \;  \frac{  4n}{\Upsilon \,\Theta \,\Gamma}$ \;
   for every $s\in V$.}
\end{equation}

Now we turn to  Case 2, and assume that $|B_x|\geq n\Gamma$.  
The approach is very similar to Case 1.
We have for $v\in A$ that
\[
     \{ u\in B_x : u\not\in \overline{N}(v) \}  \;\subseteq \;   \{u\in S^c: u\not\in \overline{N}(v)\}
   \; = \;  [v,S^c]_G \,,
\]
and hence, using Equation  (\ref{eq.Adef2}),
\[      |[v,B_x]_{\overline{G}}|   \;\geq \;   |B_x|\,-\,  |[v,S^c]_G|   \;\geq \; n\Gamma-M\epsilon |S^c| \;\geq \; n\Gamma \,-\ nM\epsilon.
\]
This gives the following analogue of  Equation (\ref{eq.PwAxbound}):
\begin{equation*}
      \overline{P}(v,B_x)      \; \geq \; \Gamma \hspace{6mm} \text{for all $v\in A$. }     
\end{equation*}
Similarly to Case 1, we deduce 
\[   
     \overline{P}^{(2)}(v,x)  \;\geq \;  \frac{\Gamma}{n}   \hspace{6mm} \forall \, v\in A      
\]
and, using Equation (\ref{eq.PwAbound}), 
\begin{equation*}
      \overline{P}^{(3)}(u,x)  \;\geq \;  \frac{\Theta\,\Gamma}{n}   \hspace{8mm}\text{for all }u\in B.
\end{equation*}
Therefore
\begin{equation*}
    \overline{P}(T_x\leq 3 \,|\, X_0=w)  \;>\;  \frac{\Theta\,\Gamma}{n}    \hspace{7mm}\text{for all }w\in A\cup B, \quad\text{and}
\end{equation*}
\begin{equation*}
     \overline{P}(T_x \leq 4  \,|\,X_0=s)   \;\geq \;  \frac{\Upsilon\,\Theta\,\Gamma}{n} \hspace{6mm}\text{for every }s\in V.
\end{equation*}      
As in Case 1, we conclude that
\begin{equation*}
   \text{If Case 2 holds, then \; $\overline{\mathbb{E}}(T_x\,|\,X_0=s)  \;\leq \;  \frac{  4n}{\Upsilon \,\Theta\,\Gamma}$ \;
   for every $s\in V$.}
\end{equation*}       

Combining Cases 1 and 2 tells us that 
\begin{equation*}
   \overline{\mathbb{E}}(T_x\,|\,X_0=s)  \;\leq \;   
   \frac{  4n}{\Upsilon \,\Theta \,\Gamma}  
   \hspace{5mm}  \text{for every $s\in V$.}         
\end{equation*}    
Since the upper bound is independent of $x$, we conclude that 
\[   \mathcal{K}(\overline{G})   \;\leq \;  n  \Psi_+\,  
    \hspace{5mm}\text{where } \Psi_+ \;=\;   \frac{4}{\Upsilon\,\Theta \,\Gamma } \,.
\] 
As explained near the beginning of this proof, the proposition now follows upon defining $\Psi_{L,U}$ to be $\max\{ \Psi_-,\Psi_+\}$.
\hfill  $\Box$

\bigskip
\begin{remark}
Upon chasing through our inequalities (specifically (\ref{eq.Cbound}), (\ref{eq.Mepsbd}), (\ref{eq.U2MLCbound}), (\ref{eq.Gammabound}), and the line following (\ref{eq.PzABbound})), 
one finds that $\Psi_+\leq 512\, L^{-1}(1-U)^{-2}$, which is less 
than $\Psi_-=2\,\epsilon^{-2}$.  So the above proof actually gives the value $\Psi_{L,U}=2\,\epsilon^{-2}$.
\end{remark}

\section{Kemeny's constant for the join of two graphs}\label{sec:join}

Let $H_1$ and $H_2$ be two graphs (not necessarily connected) 
with disjoint vertex sets $V(H_1)$ and $V(H_2)$.  
The \textit{join} $H_1\vee H_2$
is the graph obtained from the disjoint union of $H_1$ and $H_2$ by adding all of the edges joining a vertex of $H_1$ to a vertex of $H_2$.

\begin{theorem}
     \label{thm.joinbound}
Let $H_1$ and $H_2$ be graphs with $n_1$ and $n_2$ vertices respectively.  
Let $G=H_1\vee H_2$ be their join, and let $n=n_1+n_2$.  
Then $\mathcal{K}(G)  \,\leq \, 3n$.   
\end{theorem}

\bigskip
\noindent
\textbf{Proof:}  
For $i\in\{1,2\}$, let $V_i$ be the set of vertices in $H_i$. Also  
let $V=V_1\cup V_2$.

We shall use the Markov chain notation of Section \ref{sec:notation} 
for the random walk on $G$.
Recall in particular that for a vertex $x$, the random variable $T_x$
is the first passage time of $x$.

The proof of the theorem is based on the following claim:
\begin{verse}
\underline{Claim:}  For every $x,y\in V$,   $\mathbb{E}_y(T_x) \,\leq\,   n\,+\,2\,\mathbb{E}_x(T_x)$.
\end{verse}
Indeed, suppose the claim is true. Let $y\in V$. Then 
\begin{align*}
    \mathcal{K}(G)  \; & =\;  \sum_{x\in V\setminus \{y\}} 
    \pi(x) \,\mathbb{E}_y(\tau_x)  
    \\
    & \leq \;  \sum_{x\in V}  \pi(x)  \left( n \,+\,\,\frac{2}{\pi(x)} \right)   \hspace{5mm}  \hbox{(by Equation (\ref{eq.piprop}))}
        \\
    &    = \;   n\,+\,2n \,,
\end{align*}
which proves the theorem.

Now we shall prove the claim.  Let $x,y\in V$.  The statement of the claim is trivial if $x=y$, so assume $x\neq y$.

Without loss of generality, assume $x\in V_1$.  

Let $T[V_2]$ and $\theta$ be the random times defined by 
\begin{align*}  
  T[V_2]  \; & =\;  \min\{ t\geq 0\,:\, X_t\in V_2\} \hspace{5mm}\text{and}
   \\
   \theta \; & = \;  \min\{s\geq T[V_2]\,:  \, X_{s+1}\in V_1 \} \,.
\end{align*}
That is, $T[V_2]$ is the first time that the random walk visits $H_2$, and
 $\theta$ is the time immediately before the first step from $H_2$ to $H_1$.
We observe that
\[    \theta \;=\;  \min\{  s\geq 0:  X_s\in V_2 \text{ and }X_{s+1}\in V_1\} \,.
\]

The following observation about $\theta$ is key.  Since every vertex of $H_2$ is connected to every vertex of $H_1$, 
and since each step of the Markov chain chooses among neighbours of the current vertex with equal probability, it follows
that the state of the Markov chain at time $\theta+1$ is \textit{equally likely to be any of the $n_1$ vertices of $H_1$}.
More formally, if $k\geq 0$  and $D$ is any event that depends only on $\{X_0,X_1,\ldots,X_k\}$, then we have
\begin{equation}
   \label{eq.PDtheta}
     P_y(   X_{k+1}=v\,|\, D \text{ and }\theta=k)  \;=\;  \frac{1}{n_1}  \hspace{5mm}  \forall \, v\in V_1.
\end{equation}
Furthermore, the Markov property implies that 
\begin{equation}
 \nonumber 
   \mathbb{E}_y\left(  T_x \,|\,\theta<T_x, \,\theta=k, \, X_{k+1}=v \right)  \;=\;  
     \begin{cases}  k+1  & \text{if }v=x   \\
          k+1 + \mathbb{E}_v(T_x)  & \text{if }v\in V_1\setminus\{x\},
          \end{cases}
\end{equation}
and hence (using Equation (\ref{eq.PDtheta}) with $D$ being the event that $\theta<T_x$)
\begin{equation}
    \label{eq.Eytauxmarkov}
    \mathbb{E}_y(T_x \,|\, \theta<T_x, \,\theta=k)  \,=\,  k+1+S   \hspace{5mm}\text{where }\quad 
       S \,:=\, \frac{1}{n_1}\sum_{v\,\in \, V_1\setminus\{x\}} \!\! \mathbb{E}_v(T_x) \,.
\end{equation}

Now we shall prove 
\begin{equation}
   \label{eq.thetalesstau}
     P_y(\theta < T_x)  \;=\;   P_y(T[V_2]<T_x) \;\geq \;   \frac{1}{2}   \hspace{5mm}\forall \, y\in V.
\end{equation}
The equality in (\ref{eq.thetalesstau}) is true because the relation $\theta<T_x$ holds if and only if $T[V_2]<T_x$, 
i.e.\ if and only if the chain does not visit $x$ before it first
enters $V_2$
(we use the assumption that $x\in V_1$; note that $T[V_2]=0$ if $X_0\in V_2$).  
For the inequality, we see that $P_y(\theta<T_x)=1$ whenever 
$y\in V_2$, so assume that $y\in V_1$.  
For $k\in \mathbb{N}$, $a\in V_1$, and $b\in V_1\cup V_2$, let $\mathcal{U}_k[a\rightarrow b]$ be the set of all $k$-step walks 
$\underline{\beta}=(\beta_0,\ldots,\beta_k)$ in $G$
such that $\beta_0=a$, $\beta_k=b$, and $\beta_i\in V_1\setminus \{b\}$ for all $i=1,\ldots,k-1$.  
That is, a walk in $\mathcal{U}_k[a\rightarrow b]$ has its first visit to  $b$ (or its first return, if $a=b$) at time $k$, and
does not visit $H_2$ before this time. 
In particular, for any walk in this set we must have $T[V_2]\geq k$, with equality possible only if $b\in V_2$.
Moreover, we have
\[
    P_a(T_b<T[V_2] \hbox{ and }T_b=k)  \;=\;   \sum_{\underline{\beta}\, \in\, \mathcal{U}_k[a\rightarrow b]} P_a(\underline{\beta})
       \hspace{12mm}\text{for }a,b\in V_1.
\]
Fix $u\in V_2$.  We now define a function $f$ from $\bigcup_{k=1}^{\infty}\mathcal{U}_k[y\rightarrow x]$ to 
$\bigcup_{k=1}^{\infty}  \mathcal{U}_k[y\rightarrow u]$ as follows.
For each $k\geq 1$ and $\underline{\beta}=(\beta_0,\ldots,\beta_k)\in \mathcal{U}_k[y\rightarrow x]$, let
$f(\underline{\beta})=(\beta_0,\ldots,\beta_{k-1},u)$.
We know that $f(\underline{\beta})$ is a walk in $G$ because 
$\beta_{k-1}\in V_1$  and $u\in V_2$.
Observe that $f$ is one-to-one and that 
$P_y(\underline{\beta}) \,=\, P_y\left(f(\underline{\beta})\right)$ for every $\underline{\beta}$.
Therefore 
\begin{align*}
  P_y(T_x=k<T[V_2])  \;  & = \; \sum_{\underline{\beta}\,\in\, \mathcal{U}_k[y\rightarrow x]} P_y(\underline{\beta})
    \\
    & = \;  \sum_{\underline{\beta}\,\in\, \mathcal{U}_k[y\rightarrow x]} P_y(f(\underline{\beta}))    
      \\
    & \leq   \;  P_y(T_u=T[V_2]=k<T_x )
    \\
     & \leq  \;  P_y(T[V_2]=k<T_x ) \,.
\end{align*}
Summing the resulting inequality over $k$ shows that 
\[    P_y(T_x<T[V_2])  \;\;  \leq \; \;  P_y(T[V_2]<T_x ) \,,
\]
and Equation (\ref{eq.thetalesstau}) follows.

Next we shall show that 
\begin{equation}
     \label{eq.Ethetabound}
     \mathbb{E}_v(\theta)  \;\leq\; n-1   \hspace{5mm}\text{for every }v\in V.
\end{equation}
Notice that for every $u\in V_2$, $P(u,V_1)=n_1/d_u  \,\geq \, n_1/(n-1)$.  Therefore, for $v\in V_2$, we have
$P_v(\theta\geq 1) \,=\,1-P(v,V_1) \,\leq\, 1-\frac{n_1}{n-1}$ and, for every $k\in \mathbb{N}$, 
$P_v(\theta\geq k+1\,|\,\theta\geq k)\,\leq \,1-\frac{n_1}{n-1}$.
By induction, we then see that
\[     P_v(\theta \geq k) \;\leq \;  \left(1-\frac{n_1}{n-1}\right)^{k}  \hspace{5mm}\forall \,k\in \mathbb{Z}^+,
\]
and hence
\begin{equation}
   \label{eq.EvthetaH2}
    \mathbb{E}_v(\theta)  \;=\;  \sum_{k=1}^{\infty}  P_v(\theta\geq k)    \;\leq \;  \frac{  1-\frac{n_1}{n-1}}{ 1- \left(1-\frac{n_1}{n-1}\right)}  \;=\; \frac{n-1}{n_1}   - 1
    \hspace{6mm}\text{if }v\in V_2.
\end{equation}
Now suppose $v\in V_1$. 
Then a similar argument to the above (but using the relation 
$P_v(T[V_2]\geq 1)=1$) shows that $\mathbb{E}_v(T[V_2])\,\leq \,(n-1)/n_2$ for the case that $v\in V_1$.  
Now consider splitting the time interval from 0 to $\theta$ into the 
part before $T[V_2]$ and the part after $T[V_2]$.
With this viewpoint, and using Equation (\ref{eq.EvthetaH2}), it is not hard to see that 
\begin{align*}
    \mathbb{E}_v(\theta)  \; & =\; \mathbb{E}_v(T[V_2])  \,+\,  
    \sum_{w\,\in\, V_2}P_v(X_{T[V_2]}=w) \, \mathbb{E}_w(\theta)
    \\
    & \leq  \; \frac{n-1}{n_2} \,+\, \frac{n-1}{n_1} -1
    \quad  \leq \;  (n-1)\left( 1+\frac{1}{n-1}  \right) -1
    \quad  =  \; n-1
    \hspace{7mm}\text{if }v\in V_1.
\end{align*}
This completes the proof of Equation (\ref{eq.Ethetabound}).    

From Equation (\ref{eq.Eytauxmarkov}), we see that $\mathbb{E}_x(T_x\,|\,\theta<T_x,\theta=k)  \,\geq \,S$ for every $k\geq 0$.
Therefore $\mathbb{E}_x(T_x\,|\, \theta<T_x)  \,\geq \,S$, and
\begin{align*}
    \mathbb{E}_x(T_x)  \;& = \;  \mathbb{E}_x(T_x\,|\,\theta<T_x)  \,P_x(\theta<T_x)  \,+\,  \mathbb{E}_x(T_x\,|\,\theta>T_x)  \,P_x(\theta>T_x) 
    \\
      & \geq \;  S \,\left(\frac{1}{2}\right)  \,+\, 0     \hspace{5mm}\text{(by Equation (\ref{eq.thetalesstau}))} ,
\end{align*}
from which it follows that 
\begin{equation}
   \label{eq.Sbound}
     S  \;\leq \;  2\,\mathbb{E}_x(T_x)    \,.
\end{equation}

We are almost done.  For $y\neq x$ we have
\begin{align}
   \nonumber
    \mathbb{E}_y(T_x \,|\, \theta<T_x)  \,& = \; \sum_{k=1}^{\infty} \mathbb{E}_y(T_x \,|\, \theta<T_x, \,\theta=k)  \,P_y( \theta=k\,|\, \theta<T_x)
    \\
    \nonumber
    & = \; \sum_{k=1}^{\infty}( k+1+S )  \,P_y( \theta=k\,|\, \theta<T_x)    \hspace{5mm}\text{(by Eq.\ (\ref{eq.Eytauxmarkov}))}
      \\
       \label{eq.Eytausecond}  
      & =  \;   \mathbb{E}_y(\theta\,|\, \theta<T_x)  \,+1+S \,.
\end{align}      
Finally, we have
\begin{align*}
          \mathbb{E}_y(T_x)  \;& =\;    
           \mathbb{E}_y(T_x\,|\,\theta>T_x)\,P_y(\theta>T_x) \,+\,  \mathbb{E}_y(T_x \,|\,\theta<T_x)\,P_y(\theta<T_x)   
          \\
            & \leq \;       
             \mathbb{E}_y(\theta \,|\,\theta>T_x)\,P_y(\theta>T_x)  \,+\,     \left( \mathbb{E}_y(\theta \,|\,\theta<T_x)   \,+\, 1+S \right) \,P_y(\theta<T_x)         
             \\
             & 
                \hspace{25mm}\text{(using the condition 
             $\{\theta>T_x\}$, as well as Eq.\ (\ref{eq.Eytausecond}))}
            \\
          &  =   \;  \mathbb{E}_y(\theta) \,+\,(1+S)\,P_y(\theta<T_x)    
          \\
          &  \leq \;  (n-1) \,+\,1 \,+\,  2\,\mathbb{E}_x(T_x)
                    \hspace{15mm}\text{(by Eqs.\ (\ref{eq.Ethetabound}) and (\ref{eq.Sbound}))}.   
\end{align*}            
This proves the claim, and the theorem follows.
\hfill  $\Box$            

\section{Applications}\label{sec:app}
In this subsection, we present various families of graphs in terms of Nordhaus-Gaddum problem.

\subsection{Regular graphs}

Let $G$ be a $k$-regular graph with $n$ vertices. Here we recall the Kirchhoff index of $G$: it is given by $\Kf(G) = \frac{1}{2}\mathbf{1}^TR\mathbf{1}$. Since $G$ is regular, we see from \eqref{eq:res} that
\begin{align*}
	\mathcal{K}(G) = \frac{k\Kf(G)}{n}.
\end{align*}
From \cite[Proposition~4]{palacios2010kirchhoff}, we have $\frac{n(n-1)}{2k}\leq \Kf(G)\leq \frac{3n^3}{k}$. Hence,
\begin{align}\label{bounds for regular}
	\frac{n-1}{2}\leq \mathcal{K}(G)\leq 3n^2.
\end{align}
Hence, $\mathcal{K}(G) = O(n^2)$. By Corollary~\ref{cor.regular}, $\min\{\mathcal{K}(G),\mathcal{K}(\overline{G})\} = O(n)$ and so $\max\{\mathcal{K}(G),\mathcal{K}(\overline{G})\} = O(n^2)$. Therefore,
$$\mathcal{K}(G)\mathcal{K}(\overline{G}) = O(n^3).$$


\subsubsection{Construction of regular graphs $G$ with $\mathcal{K}(G) = \Theta(n^2)$}
In this subsection, we will provide families of regular graphs for which the growth rate of Kemeny's constant is $\Theta(n^2)$. An obvious example is the $n$-cycle whose Kemeny's constant is given \cite{kim2022families} by $\frac{1}{6}(n^2-1)$. In addition to that, we will show two ways of constructing them.

Let $G$ be a $k$-regular graph $G$ with $n$ vertices, and $\mu_1,\mu_2,\dots,\mu_n$ be the eigenvalues of the adjacency matrix of $G$ in non-increasing order. Then $\mu_1=k$. Then Kemeny's constant can be obtained from \eqref{eq.Keigensum} as follows:
\begin{align}\label{Kemeny:spec reg}
	\mathcal{K}(G) = \sum_{i = 2}^{n}\frac{k}{k-\mu_i}.
\end{align}
If some of eigenvalues are close to $k$ so that they are enough to determine the order of Kemeny's constant, then it is unnecessary to understand the remaining eigenvalues. From this, we shall use the so-called \textit{equitable partition} (see \cite{godsil2001algebraic}) for the construction, which induces the so-called \textit{quotient matrix} which is a smaller matrix whose eigenvalues belong to the spectrum of the adjacency matrix of $G$.

Let $n_2\geq 2$, and $G_1,\dots,G_{n_2}$ be $k_1$-regular graphs with $n_1$ vertices. Given a $k_2$-regular graph $H$ with $n_2$ vertices, we define $G$ to be a graph from the disjoint union of $G_1,\dots,G_{n_2}$ by inserting edges as follows: for vertices $i$ and $j$ of $H$, if $i$ and $j$ are adjacent, then we add non-adjacent $n_1$ edges for which each edge joins a vertex of $G_i$ and a vertex of $G_j$ (that is, those $n_1$ edges form a matching in $G$); and if $i$ and $j$ are not adjacent, then there is no edge between $G_i$ and $G_j$. Note that $G$ is $(k_1+k_2)$-regular. Then $(V(G_1),\dots,V(G_{n_2}))$ is an equitable partition, and so the quotient matrix $Q$ is given by $$Q = k_1I +A(H)$$
where $A(H)$ is the adjacency matrix of $H$. Then the eigenvalues of $Q$ are given by $k_1+\theta_i$ for $1\leq i\leq n_2$ where $\theta_1,\dots,\theta_{n_2}$ are the eigenvalues of $A(H)$ in non-increasing order. Then $k_2 = \theta_1$. Hence,
\begin{align*}
	\mathcal{K}(G)>\sum_{i=2}^{n_2}\frac{k_1+k_2}{k_2-\theta_i} = \frac{k_1+k_2}{k_2}\sum_{i=2}^{n_2}\frac{k_2}{k_2-\theta_i} = \left(\frac{k_1}{k_2}+1\right)\mathcal{K}(H).
\end{align*}
Therefore, if $\mathcal{K}(H) = \Theta(n_2^2)$ with $n_2 = \Theta(n)$, then $\mathcal{K}(G) = \Theta(n^2)$ follows from \eqref{bounds for regular}.

Now we introduce a different construction, using Kemeny's constant of a graph with bridges in \cite{breen2022kemeny}. This generalizes the necklace graph in \cite[]{breen2023kemeny} (see Figure~\ref{Figure:necklace} for the construction of the necklace graph), and also it will be used in the next subsection. We recall some definition and notation.

\begin{definition}\cite{breen2022kemeny}
	Let $\cT$ be a tree on $d$ vertices where $V(\cT)=\{1,\dots,d\}$. Let $G_1,\dots,G_d$ be connected graphs. Let $G$ be a graph constructed as follows: the vertices $1,\dots,d$ are replaced by the graphs $G_1,\dots,G_d$, respectively; and if $\{i,j\}$ is an edge of $\cT$, then some vertex $v_i\in V(G_i)$ is chosen, and some vertex $v_j\in V(G_j)$ is chosen, and the two vertices are joined with an edge so that $\{v_i, v_j\}$ is a bridge in $G$. Then $G$ is said to be a \textit{chain of $G_1,\dots,G_d$ with respect to $\cT$}. We denote by $\mathcal{B}_G$ the set of the $(d-1)$ bridges, used in the construction of $G$, that correspond to the edges of $\cT$.
\end{definition}

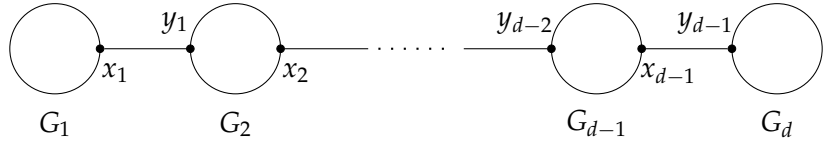
\begin{figure}[h!]
	\begin{center}
		\begin{tikzpicture}
			\tikzset{enclosed/.style={draw, circle, inner sep=0pt, minimum size=.10cm, fill=black}}
			\node[enclosed, xshift=0.6cm, label={below, xshift=.2cm : $x_1$}] (w1) at (-4.8,0) {};
			\node[enclosed, xshift=-0.6cm, label={above, xshift=-.2cm : $y_1$}] (v2) at (-2.4,0) {};
			\node[enclosed, xshift=0.6cm, label={below, xshift=.2cm : $x_2$}] (w2) at (-2.4,0) {};
			\node[] (v3) at (-.5,0) {}; 
			
			\node[] (wl-2) at (.5,0) {};
			\node[enclosed, xshift=-0.6cm, label={above, xshift=-.35cm : $y_{d-2}$}] (vl-1) at (2.4,0) {};
			\node[enclosed, xshift=0.6cm, label={below, xshift=.35cm : $x_{d-1}$}] (wl-1) at (2.4,0) {};
			\node[enclosed, xshift=-0.6cm, label={above, xshift=-.35cm : $y_{d-1}$}] (vl) at (4.8,0) {};

			\draw (-4.8,0) circle (0.6cm);
			\draw (-2.4,0) circle (0.6cm);
			\draw[thick, loosely dotted] (-.5,0)--(.5,0);
			\draw (2.4,0) circle (0.6cm);
			\draw (4.8,0) circle (0.6cm);
			
			\draw (w1) -- (v2); \draw (w2)--(v3);
			\draw (wl-2) -- (vl-1); \draw (wl-1)--(vl);
			
			\node[] at (-4.8,-1) {$G_{1}$};
			\node[] at (-2.4,-1) {$G_{2}$};
			\node[] at (2.4,-1) {$G_{d-1}$};
			\node[] at (4.8,-1) {$G_{d}$};
		\end{tikzpicture}
	\end{center}
	\caption{An illustration of a chain of $G_1,\dots,G_d$ with respect to the path on $d$ vertices. Here $\mathcal{B}_G$ consists of edges $\{x_1, y_1\},\dots, \{x_d, y_d\}$.}\label{Figure:chain}
\end{figure}

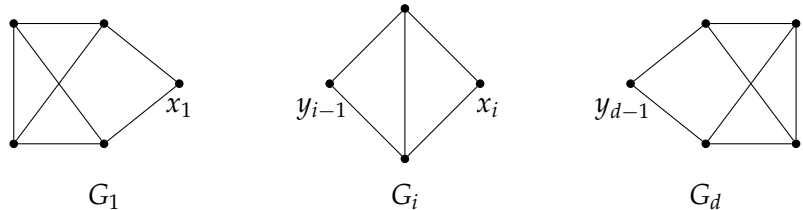
\begin{figure}[h!]
	\begin{center}
		\begin{tikzpicture}
			\tikzset{enclosed/.style={draw, circle, inner sep=0pt, minimum size=.10cm, fill=black}}
			\begin{scope}[xshift=0cm]
				\node[enclosed, label={left, yshift=0cm: }] (v_1) at (-1.2,.8) {};
				\node[enclosed, label={left, yshift=0cm:}] (v_2) at (-1.2,-.8) {};
				\node[enclosed, label={left, yshift=0cm: }] (v_3) at (0,.8) {};
				\node[enclosed, label={right, xshift=0cm: }] (v_4) at (0,-.8) {};
				\node[enclosed, label={below, xshift=0cm: $x_1$}] (v_5) at (1,0) {};
				
				\draw (v_1) -- (v_2);
				\draw (v_1) -- (v_4);
				\draw (v_1) -- (v_3);
				\draw (v_2) -- (v_3);
				\draw (v_2) -- (v_4);
				\draw (v_5) -- (v_3);
				\draw (v_5) -- (v_4);
				
				\node[] at (0,-1.5) {$G_{1}$};
			\end{scope}
			
			\begin{scope}[xshift=4cm]
				
				\node[enclosed, label={left, yshift=0cm: }] (v_1) at (0,1) {};
				\node[enclosed, label={left, yshift=0cm:}] (v_2) at (0,-1) {};
				\node[enclosed, label={below, xshift=-.1cm: $y_{i-1}$}] (v_3) at (-1,0) {};
				\node[enclosed, label={below, xshift=0.1cm: $x_{i}$}] (v_4) at (1,0) {};
				
				\draw (v_1) -- (v_2);
				\draw (v_1) -- (v_4);
				\draw (v_1) -- (v_3);
				\draw (v_2) -- (v_3);
				\draw (v_2) -- (v_4);
				
				\node[] at (0,-1.5) {$G_{i}$};
			\end{scope}
			
			\begin{scope}[xshift=8cm]
				\node[enclosed, label={left, yshift=0cm: }] (v_1) at (1.2,.8) {};
				\node[enclosed, label={left, yshift=0cm:}] (v_2) at (1.2,-.8) {};
				\node[enclosed, label={left, yshift=0cm: }] (v_3) at (0,.8) {};
				\node[enclosed, label={right, xshift=0cm: }] (v_4) at (0,-.8) {};
				\node[enclosed, label={below, xshift=-.1cm: $y_{d-1}$}] (v_5) at (-1,0) {};
				
				\draw (v_1) -- (v_2);
				\draw (v_1) -- (v_4);
				\draw (v_1) -- (v_3);
				\draw (v_2) -- (v_3);
				\draw (v_2) -- (v_4);
				\draw (v_5) -- (v_3);
				\draw (v_5) -- (v_4);
				
				\node[] at (0,-1.5) {$G_{d}$};
			\end{scope}
		\end{tikzpicture}
	\end{center}
	\caption{Taking these $G_i$, $x_i$, $y_i$ for $1\leq i\leq d$ in Figure~\ref{Figure:chain}, we can obtain the necklace graph in \cite{breen2023kemeny}.}\label{Figure:necklace}
\end{figure}

Let $G$ be a chain of connected graphs $G_1,\dots,G_d$ with respect to a tree $\cT$ on $d$ vertices (see Figure~\ref{Figure:chain} as an example when $\cT$ is a path). For each bridge $\{x, y\}\in\mathcal{B}_G$, we have exactly two components in $G\backslash \{x, y\}$, where $G\backslash \{x, y\}$ is the graph obtained from $G$ by removing the edge $\{x, y\}$. We use $W_x$ (resp. $W_y$) to denote the number of edges of the component with $x$ (resp. $y$) in $G\backslash \{x, y\}$. Let $\overline{W}_x=m_G-W_x$ and $\overline{W}_y=m_G-W_y$.

We then see from \cite[Theorem 3.16]{breen2022kemeny} that 
\begin{align*}
	\mathcal{K}(G) > \sum\limits_{\{x, y\} \in \mathcal{B}_G}\frac{(2\overline{W}_x-1)(2\overline{W}_y-1)}{2m_G}:=\Gamma(G).
\end{align*}

Let $k$ be an odd integer greater than $2$. Let $\mathcal{H}_1^k$ be the set of graphs with degree sequence $(k,\dots,k,k-1)$, and let $\mathcal{H}_2^k$ be the set of graphs with degree sequence $(k,\dots,k,k-1,k-1)$. 

Let $\cT$ be a path on $d$ vertices. Given $G_1,G_d\in \mathcal{H}_1$ and $G_2,\dots,G_{d-1}\in \mathcal{H}_2$, we let $G$ be a chain of connected graphs $G_1,\dots,G_d$ with respect to $\cT$ such that for each bridge $\{x_i, y_i\}$ in $\mathcal{B}_G$, $1\leq i\leq d$, $\mathrm{deg}(x_i) = \mathrm{deg}(y_i) = k$ (by properly choosing $x_i$'s and $y_i$'s as vertices with degree $k-1$ in $G_1,\dots,G_d$). Suppose that $n_1=|V(G_1)| = |V(G_d)|$ and $n_2=|V(G_2)|=\cdots =|V(G_{d-1})|$. Then
\begin{align*}
	2\overline{W}_{x_i} = (d-i-1)n_2k + n_1k+1\;\;\text{and}\;\;2\overline{W}_{y_i} = (i-1)n_2k + n_1k+1.
\end{align*}
We can find that
\begin{align*}\label{temp;eqn}
	\Gamma(G)=&~\frac{k(d-1)}{(2n_1+(d-2)n_2)}\left((n_1-n_2)^2+dn_2(n_1-n_2)+\frac{1}{6}(d^2+d)n_2^2\right)\\
			 =&~\frac{k(d-1)}{6(2n_1+(d-2)n_2)}\left(6n_1^2+6(d-2)n_1n_2+(d-3)(d-2)n_2^2\right)\\
            >&~\frac{k(d-1)}{6(2n_1+(d-2)n_2)}\left(4n_1^2+4(d-2)n_1n_2+(d-2)(d-2)n_2^2-(d-2)n_2^2\right)\\
            =&~\frac{k(d-1)(2n_1+(d-2)n_2)}{6}-\frac{k(d-1)}{6(2n_1+(d-2)n_2)}(d-2)n_2^2.
\end{align*}
Let $n=|V(G)|$. Since $n = 2n_1+(d-2)n_2$, it follows that
\begin{align*}
	\mathcal{K}(G)=\Omega(kdn).
\end{align*}
We see that the order of Kemeny's constant is independent of $n_1$ and $n_2$. Since $k<n_1$ and $k<n_2$, we have $k<\frac{n}{d}$. Therefore, if $kd = \Theta(n)$ then $\mathcal{K}(G)= \Theta(n^2)$. As an example, since the necklace graph has $k = 3$ and $d = \frac{n-2}{4}$, its Kemeny's constant is $\Theta(n^2)$.
	
\subsection{Barbell-type graphs}

It appears in \cite{breen2019computing} that the graph obtained from a path on $d$ vertices by appending a clique of size $d$ to each end-vertex attains the largest order of Kemeny's constant. Here we will show that it is enough to append a graph with sufficiently many edges to each end-vertex in order to obtain the same result.

Let $\cT$ be a path on $d$ vertices, and $G_1,\dots,G_d$ be connected graphs. Let $G$ be a chain of connected graphs $G_1,\dots,G_k$ with respect to $\cT$. Suppose that $n = |V(G)|$, $d=|V(G_1)| = |V(G_d)|$ and $1=|V(G_2)|=\cdots =|V(G_{d-1})|$. Assume that $d = \Theta(n)$, $m_{G_1}=\Theta(d^2)$ and $m_{G_d}=\Theta(d^2)$. For $i=1,\dots,d-1$, let $\{x_i, y_i\}$ be the bridge in $\mathcal{B}_G$ such that $x_i\in V(G_i)$ and $y_i\in V(G_{i+1})$. Then
\begin{align*}
	\overline{W}_{x_i} = m_{G_d}+d-i\;\;\text{and}\;\;\overline{W}_{y_i} = m_{G_1}+i.
\end{align*}
We can find that
\begin{align*}
	\sum\limits_{\{x, y\} \in \mathcal{B}_G}\frac{(2\overline{W}_x-1)(2\overline{W}_y-1)}{2m_G}>\sum_{i=1}^{d-1}\frac{4m_{G_1}m_{G_d}}{2(m_{G_1}+m_{G_d}+d-1)}.
\end{align*}
Hence, $\mathcal{K}(G) = \Theta(n^3)$. Since $\overline{G}$ has a vertex of degree $n-3$, we have $\mathcal{K}(\overline{G}) = O(n)$. Therefore,
$$\mathcal{K}(G)\mathcal{K}(\overline{G}) = \Theta(n^4).$$

\begin{remark}
	Taking $G_1$ and $G_{d}$ as the complete bipartite graph $K_{\lfloor\frac{d}{2}\rfloor,\lceil\frac{d}{2}\rceil}$, we can construct a bipartite graph $G$ with $\mathcal{K}(G)\mathcal{K}(\overline{G}) = \Theta(n^4)$.
\end{remark}

\subsection{Trees}

Let $\cT$ be a tree on $n$ vertices. Since there exists a vertex of degree $n-2$ in the complement of $\cT$, we have $\mathcal{K}(\overline{\cT})= O(n)$. Now we consider the order of $\mathcal{K}(\cT)$. The \textit{Wiener index}, denoted as $W(\cT)$, of $\cT$ is the sum of distances for all pairs of two distinct vertices. The following appears in \cite{jang2023kemeny}:
\begin{align*}
	\mathcal{K}(\cT) = \frac{2W(\cT)}{n-1}-n+\frac{1}{2}.
\end{align*}
Hence, trees with the maximum Wiener index attains the maximum Kemeny's constant. It is known in \cite{entringer1976distance} that when $\cT$ is the path, the maximum of Wiener index is attained, which is $O(n^3)$. Hence, $\mathcal{K}(\cT) = O(n^2)$. Therefore,
$$\mathcal{K}(\cT)\mathcal{K}(\overline{\cT}) = O(n^3).$$

\subsection{Strongly regular graphs}
In this subsection, we refer the reader to \cite[Section 10]{godsil2001algebraic} for the comprehensive definition and all relevant properties related to strongly regular graphs that we shall use here.

Let $G$ be a connected strongly regular graph with parameter $(n,k;a,c)$. Then the spectrum of $G$ is given by
$$\left\{k, \left(\frac{a-c+\sqrt{(a-c)^2+4(k-c)}}{2}\right)^{m_1},\left(\frac{a-c-\sqrt{(a-c)^2+4(k-c)}}{2}\right)^{m_2}\right\}$$
where $m_1 = \frac{1}{2}\left(n-1-\frac{2k+(n-1)(a-c)}{\sqrt{(a-c)^2+4(k-c)}}\right)$ and $m_2 = \frac{1}{2}\left(n-1+\frac{2k+(n-1)(a-c)}{\sqrt{(a-c)^2+4(k-c)}}\right)$. One can find from \eqref{Kemeny:spec reg} that
\begin{align*}
	\mathcal{K}(G) = \frac{(n-2)k^2-(n-1)(a-c)k}{k^2-(a-c+1)k +c}.
\end{align*}
It is known that $k^2 = nc+(a-c)k+k-c$. Hence
\begin{align*}
	\mathcal{K}(G) = \frac{(n-2)(nc+(a-c)k+k-c)+(n-1)(c-a)k}{nc} = O\left(n\right).
\end{align*}
The complement $\overline{G}$ of $G$ is also a strongly regular graph with parameter $(n,n-k-1;n-2-2k+c,n-2k+a)$. It is also known that $\overline{G}$ is disconnected if and only if $\overline{G}$ is $m$ copies of a complete graph for some $m>1$. Therefore, if $\overline{G}$ is not $m$ copies of a complete graph, then
$$\mathcal{K}(G)\mathcal{K}(\overline{G}) = \Theta(n^2).$$

\subsection{Distance regular graphs with growing diameter}

Considering Remark \ref{remark: n^3} and Proposition~\ref{prop:fixed diameter}, one might expect that the order of Kemeny's constants is higher as the diameter of the graphs grows. For instance,  the cycle of length $n$ has diameter $\lfloor \frac{n}{2}\rfloor$ and its Kemeny's constant is $\Theta(n^2)$. However, in this subsection, we introduce certain families of examples that display contrasting behavior.

We consider distance regular graphs with classical parameters and we refer the reader to \cite{brouwer2012distance} for a comprehensive monograph on distance regular graphs and to \cite{jurivsic2017restrictions} for spectral properties of distance regular graphs with classical parameters. Note that these references use $v$ for the number of vertices while we denote it by $n$. 

Let $G$ be a distance regular graph with classical parameters $(d,b,\alpha,\beta)$ where $d$ is the diameter of $G$. Let $k$ be the degree of a vertex. It is known that there are $d+1$ distinct eigenvalues $\theta_0,\dots,\theta_d$. We define $[i]_b = 1+b+\cdots+b^{i-1}$ for $i\geq 1$, and $[0]_b=0$. It is found in \cite[Lemma 2]{jurivsic2017restrictions} that for $0\leq i\leq d$,
\begin{align*}
    \theta_i = [d-i]_b(\beta -\alpha [i]_b)-[i]_b.
\end{align*}
If $b>0$ then the eigenvalues are given in decreasing order. The multiplicity of $\theta_i$ is given by
\begin{align*}
    m_i = \frac{(1+\alpha[d-2i]_b+b^{d-2i}\beta)\prod_{j=0}^{i-1}\alpha_j}{(1+\alpha[d]_b+b^d\beta)\prod_{j=1}^i\beta_j},
\end{align*}
where
\begin{equation*}
    \begin{aligned}
        \alpha_j =&~ b[d-j]_b(\beta-\alpha[j]_b)(1+\alpha [d-j]_b+b^{d-j}\beta), && (0\leq j\leq d-1);\\
    \beta_j =&~ [j]_b(\beta-\alpha[j]_b+b^j)(1+\alpha[d-j]_b), && (1\leq j\leq d).
    \end{aligned}
\end{equation*}

Suppose that $b>1$ and $(b-1)\beta+\alpha>0 $. We see that $k=\theta_0=\beta[d]_b$ and $\theta_1=[d-1]_b(\beta-\alpha)-1$. Then,
$$1- \frac{\theta_1}{k} = 1 - \frac{[d-1]_b(\beta-\alpha)-1}{[d]_b\beta}\rightarrow \frac{(b-1)\beta+\alpha}{b\beta}\quad \quad \text{as $d\rightarrow\infty$.}$$
Therefore, $\frac{k}{k-\theta_1}=O(1)$ and so 
$$\mathcal{K}(G) \leq \frac{(n-1)k}{k-\theta_1} = O(n).$$

\begin{example}
    Kemeny's constants for families (C2), (C3), (C3a), (C4), (C4a), (C10), (C11), and (C11a) in \cite[Tables 6.1 and 6.2]{brouwer2012distance} are $O(n)$ while their diameters grow as $n$ increases.
\end{example}

An example of a family with $b=1$ is the set of Hamming graphs (see \cite{brouwer2012distance}). Then $n=q^d$, $b=1$, $\alpha=0$, and $\beta = q$. Using the formulae for eigenvalues and their multiplicities, and the fact that $\frac{1}{j}\leq \frac{2}{j+1}$ for $j\geq 1$, we can find that 
\begin{align*}
    \mathcal{K}(G) = &~\frac{d}{q}\sum_{j=1}^d \binom{d}{j}\frac{(q-1)^{j+1}}{j} \\
    \leq&~  \frac{2d}{q}\sum_{j=1}^d \binom{d}{j}\frac{(q-1)^{j+1}}{j+1} = \frac{2d}{q}\left(\frac{1}{d+1}q^{d+1}-(q-1)\right) = O(n).
\end{align*}

\section{Discussion} \label{sec:concl}
Let $G$ be a connected graph. We have seen in Remark~\ref{remark: n^3} that $\mathcal{K}(G) = O(n^3)$. Trivially, $\mathcal{K}(G)+\mathcal{K}(\overline{G}) = O(n^3)$. Consequently, we have examined $\mathcal{K}(G)\mathcal{K}(\overline{G})$ for the Nordhaus-Gaddum Problem in relation to Kemeny's constant. We have proved that when maximum degree is $n-\Omega(n)$, or when it is $n-O(1)$, we have $\mathcal{K}(G)\mathcal{K}(\overline{G}) = O(n^4)$. 
Interchanging the roles of $G$ and $\overline{G}$, we see also that 
$\mathcal{K}(G)\mathcal{K}(\overline{G}) = O(n^4)$ if the minimum
degree is either $O(1)$ or $\Omega(n)$.
To completely resolve the issue, future work should consider graphs in which the maximum degree is $n-o(n)$ and $n-\omega(1)$ \textit{and} the minimum degree is $\omega(1)$ and $o(n)$.
Additionally, we have presented calculations for various families of related graphs.

The above-mentioned $O(n^3)$ bound shows that 
$\mathcal{K}(G)\mathcal{K}(\overline{G})=O(n^6)$.  
Although we do not 
expect that this bound to be optimal, we do not have a better 
bound for all graphs.  The final sentence of Corollary
\ref{cor-deltaplusbound} implies that 
$\mathcal{K}(G)\mathcal{K}(\overline{G})=O(n^5)$ whenever
$\Delta(G)+\delta(G) \neq n-1$ (if $\Delta(G)+\delta(G)<n-1$,
then $\Delta(\overline{G})+\delta(\overline{G})\geq n$).  However, we
do not even have examples to show that the worst case of 
$\mathcal{K}(G)\mathcal{K}(\overline{G})$ is not $O(n^4)$.

An alternative approach to the Nordhaus-Gaddum problem can be based on graph diameter rather than maximum degree. From Proposition~\ref{prop:fixed diameter}, if $G$ is of diameter $3$, then $\mathcal{K}(G)=O(n^2)$. Since the complement of a graph of diameter more than $3$ has diameter $2$, if we understand the order of Kemeny's constant of a graph with diameter $2$, then we can address the problem.

\begin{conjecture}\label{conjecture1}
	Let $G$ be a graph with  $\diam(G)=2$. Then $\mathcal{K}(G) = O(n)$. 
\end{conjecture}

We can see from our findings that Conjecture~\ref{conjecture1} holds for several families of graphs with diameter $2$; for instance, threshold graphs, join of graphs, and strongly regular graphs.

Finally, we have not identified a family of graphs $G$ such that $\mathcal{K}(G) = \omega(n)$ and $\mathcal{K}(\overline{G})=\omega(n)$. Such a family would be interesting to understand the structure of $G$ together with $\overline{G}$ to have higher order of Kemeny's constant. If Conjecture~\ref{conjecture1} is proved to be true, then that family should be found among graphs with diameter $3$ whose complements also have diameter $3$.

\section*{Acknowledgement}
The authors are grateful to Jane Breen at Ontario Tech University and Steve Butler at Iowa State University for constructive conversations when this project was in its initial phase. 

\section*{Funding}
Ada Chan gratefully acknowledges the support of the NSERC Grant No. RGPIN-2021-03609. S. Kim is supported in part by funding from the Fields Institute for Research in Mathematical Sciences and from the Natural Sciences and Engineering Research Council of Canada (NSERC). S. Kirkland is supported by NSERC grant number RGPIN–2019–05408. 
N. Madras is supported in part by NSERC Grant No.\ RGPIN-2020-06124.


\begin{bibdiv}
	\begin{biblist}
		
		\bib{aouchiche2013survey}{article}{
			author={Aouchiche, Mustapha},
			author={Hansen, Pierre},
			title={A survey of {N}ordhaus-{G}addum type relations},
			date={2013},
			journal={Discrete Applied Mathematics},
			volume={161},
			number={4--5},
			pages={466\ndash 546},
		}
		
		\bib{bapat2010graphs}{book}{
			author={Bapat, Ravindra~B.},
			title={Graphs and {M}atrices},
			publisher={Springer},
			date={2010},
			volume={27},
		}
		
		\bib{breen2019computing}{article}{
			author={Breen, Jane},
			author={Butler, Steve},
			author={Day, Nicklas},
			author={DeArmond, Colt},
			author={Lorenzen, Kate},
			author={Qian, Haoyang},
			author={Riesen, Jacob},
			title={Computing {K}emeny's constant for a barbell graph},
			date={2019},
			journal={Electronic Journal of Linear Algebra},
			volume={35},
			pages={583\ndash 598},
		}
		
		\bib{breen2022kemeny}{article}{
			author={Breen, Jane},
			author={Crisostomi, Emanuele},
			author={Kim, Sooyeong},
			title={Kemeny’s constant for a graph with bridges},
			date={2022},
			journal={Discrete Applied Mathematics},
			volume={322},
			pages={20\ndash 35},
		}
		
		\bib{breen2023kemeny}{article}{
			author={Breen, Jane},
			author={Faught, Nolan},
			author={Glover, Cory},
			author={Kempton, Mark},
			author={Knudson, Adam},
			author={Oveson, Alice},
			title={Kemeny's constant for nonbacktracking random walks},
			date={2023},
			journal={Random Structures \& Algorithms},
			volume={63},
			pages={343\ndash 363},
		}
		
		\bib{brouwer2012distance}{book}{
			author={Brouwer, Andries~E.},
			author={Haemers, Willem~H.},
			title={Distance-{R}egular {G}raphs},
			publisher={Springer},
			date={2012},
		}
		
		\bib{chandra1989electrical}{inproceedings}{
			author={Chandra, Ashok~K.},
			author={Raghavan, Prabhakar},
			author={Ruzzo, Walter~L.},
			author={Smolensky, Roman},
			title={The electrical resistance of a graph captures its commute and
				cover times},
			date={1989},
			booktitle={Proceedings of the {T}wenty-{F}irst {A}nnual {ACM} {S}ymposium on
				{T}heory of {C}omputing},
			pages={574\ndash 586},
		}
		
		\bib{chen2007resistance}{article}{
			author={Chen, Haiyan},
			author={Zhang, Fuji},
			title={Resistance distance and the normalized {L}aplacian spectrum},
			date={2007},
			journal={Discrete Applied Mathematics},
			volume={155},
			number={5},
			pages={654\ndash 661},
		}
		
		\bib{das2016nordhaus}{article}{
			author={Das, Kinkar~Ch},
			author={Yang, Yujun},
			author={Xu, Kexiang},
			title={Nordhaus-{G}addum-type results for resistance distance-based
				graph invariants},
			date={2016},
			journal={Discussiones Mathematicae Graph Theory},
			volume={36},
			number={3},
			pages={695\ndash 707},
		}
		
		\bib{entringer1976distance}{article}{
			author={Entringer, Roger~C.},
			author={Jackson, Douglas~E.},
			author={Snyder, D.A.},
			title={Distance in graphs},
			date={1976},
			journal={Czechoslovak Mathematical Journal},
			volume={26},
			number={2},
			pages={283\ndash 296},
		}
		
		\bib{faught2023nordhaus}{article}{
			author={Faught, J.~Nolan},
			author={Kempton, Mark},
			author={Knudson, Adam},
			title={A {N}ordhaus-{G}addum type problem for the normalized {L}aplacian
				spectrum and graph {C}heeger constant},
			date={2023},
			journal={arXiv preprint arXiv:2304.01979},
		}
		
		\bib{GhaL}{inproceedings}{
			author={Ghayoori, Armin},
			author={Leon-Garcia, Alberto},
			title={Robust network design},
			date={2013},
			booktitle={2013 {IEEE} {I}nternational {C}onference on {C}ommunication
				({ICC}), {B}udapest},
			pages={2409\ndash 2414},
		}
		
		\bib{godsil2001algebraic}{book}{
			author={Godsil, Chris},
			author={Royle, Gordon~F.},
			title={Algebraic {G}raph {T}heory},
			publisher={Springer Science \& Business Media},
			date={2001},
			volume={207},
		}
		
		\bib{jang2023kemeny}{article}{
			author={Jang, Jihyeug},
			author={Kim, Sooyeong},
			author={Song, Minho},
			title={Kemeny's constant and {W}iener index on trees},
			date={2023},
			journal={Linear Algebra and its Applications},
		}
		
		\bib{JeSi}{article}{
			author={Jerrum, Mark},
			author={Sinclair, Alistair},
			title={Approximating the permanent},
			date={1989},
			journal={SIAM Journal on Computing},
			volume={18},
			pages={1149\ndash 1178},
		}
		
		\bib{jurivsic2017restrictions}{article}{
			author={Juri{\v{s}}i{\'c}, Aleksandar},
			author={Vidali, Jano{\v{s}}},
			title={Restrictions on classical distance-regular graphs},
			date={2017},
			journal={Journal of Algebraic Combinatorics},
			volume={46},
			pages={571\ndash 588},
		}
		
		\bib{kemeny1960finite}{book}{
			author={Kemeny, John},
			author={Snell, J.~Laurie},
			title={Finite {M}arkov {C}hains},
			publisher={Springer-Verlag, New York-Heidelberg},
			date={1976},
		}
		
		\bib{kim2022families}{article}{
			author={Kim, Sooyeong},
			title={Families of graphs with twin pendent paths and the {B}raess
				edge},
			date={2022},
			journal={Electronic Journal of Linear Algebra},
			volume={38},
			pages={9\ndash 31},
		}
		
		\bib{kirkland2016kemeny}{article}{
			author={Kirkland, Steve},
			author={Zeng, Ze},
			title={Kemeny's constant and an analogue of {B}raess' paradox for
				trees},
			date={2016},
			journal={Electronic Journal of Linear Algebra},
			volume={31},
			pages={444\ndash 464},
		}
		
		\bib{LaSo}{article}{
			author={Lawler, Gregory~F.},
			author={Sokal, Alan~D.},
			title={Bounds on the $l^2$ spectrum for {M}arkov chains and {M}arkov
				processes: {A} generalization of {C}heeger's inequality},
			date={1988},
			journal={Transactions of the American Mathematical Society},
			volume={309},
			pages={557\ndash 580},
		}
		
		\bib{levene2002kemeny}{article}{
			author={Levene, Mark},
			author={Loizou, George},
			title={Kemeny's constant and the random surfer},
			date={2002},
			journal={American Mathematical Monthly},
			volume={109},
			number={8},
			pages={741\ndash 745},
		}
		
		\bib{LPW}{book}{
			author={Levin, David~A.},
			author={Peres, Yuval},
			author={Wilmer, Elizabeth~L.},
			title={Markov {C}hains and {M}ixing {T}imes},
			publisher={American Mathematical Society},
			date={2009},
		}
		
		\bib{MooI}{article}{
			author={Moosavi, Vahid},
			author={Isacchini, Giulio},
			title={A {M}arkovian model of evolving world input-output network},
			date={2017},
			journal={PLoS ONE},
			volume={12},
			pages={e0186746},
		}
		
		\bib{nordhaus1956complementary}{article}{
			author={Nordhaus, Edward~A.},
			author={Gaddum, Jerry~W.},
			title={On complementary graphs},
			date={1956},
			journal={American Mathematical Monthly},
			volume={63},
			number={3},
			pages={175\ndash 177},
		}
		
		\bib{palacios2010kirchhoff}{article}{
			author={Palacios, Jos{\'e}~Luis},
			title={On the {K}irchhoff index of regular graphs},
			date={2010},
			journal={International Journal of Quantum Chemistry},
			volume={110},
			number={7},
			pages={1307\ndash 1309},
		}
		
		\bib{shapiro1987electrical}{article}{
			author={Shapiro, Louis~W.},
			title={An electrical lemma},
			date={1987},
			journal={Mathematics Magazine},
			volume={60},
			number={1},
			pages={36\ndash 38},
		}
		
		\bib{yang2011new}{article}{
			author={Yang, Yujun},
			author={Zhang, Heping},
			author={Klein, Douglas~J.},
			title={New {N}ordhaus-{G}addum-type results for the {K}irchhoff index},
			date={2011},
			journal={Journal of Mathematical Chemistry},
			volume={49},
			pages={1587\ndash 1598},
		}
		
		\bib{zhou2008note}{article}{
			author={Zhou, Bo},
			author={Trinajsti{\'c}, Nenad},
			title={A note on {K}irchhoff index},
			date={2008},
			journal={Chemical Physics Letters},
			volume={455},
			number={1--3},
			pages={120\ndash 123},
		}
		
	\end{biblist}
\end{bibdiv}


\end{document}